\documentclass{amsart}

\usepackage[all]{xy}

\usepackage{hyperref}

\usepackage{cleveref}

\usepackage{amssymb}
\usepackage{todonotes}
\usepackage[T1]{fontenc}
\theoremstyle{plain}
\newtheorem{theorem}[subsection]{Theorem}
\newtheorem{proposition}[subsection]{Proposition}
\newtheorem{lemma}[subsection]{Lemma}
\newtheorem{corollary}[subsection]{Corollary}
\theoremstyle{definition}
\newtheorem{definition}[subsection]{Definition}
\newtheorem{example}[subsection]{Example}

\newcommand{\Bun}{\mathrm{Bun}}
\newcommand{\crys}{\mathrm{crys}}
\newcommand{\dR}{\mathrm{dR}}
\newcommand{\et}{\mathrm{{\acute{e}t}}}
\newcommand{\F}{\mathbb{F}}

\newcommand{\Gm}{{\mathbb{G}_m}}

\renewcommand{\inf}{\mathrm{inf}}

\newcommand{\N}{\mathbb{N}}

\newcommand{\Sets}{\mathrm{Sets}}
\newcommand{\Spa}{\mathrm{Spa}}
\newcommand{\Spd}{\mathrm{Spd}}
\newcommand{\Spec}{\mathrm{Spec}}

\newcommand{\Z}{\mathbb{Z}}
\newcommand{\R}{\mathbb{R}}

\begin{document}

\title{Extending torsors on the punctured $\Spec(A_\inf)$}
\author{Johannes Ansch\"{u}tz}
\email{ja@math.uni-bonn.de}
\date{\today}

\begin{abstract}
  We prove that torsors under parahoric group schemes on the punctured spectrum of Fontaine's ring $A_\inf$, extend to the whole spectrum. Using descent we can extend a similar result for the ring $\mathfrak{S}$ of Kisin and Pappas to full generality. Moreover, we treat similarly the case of equal characteristic. As applications  we extend results of Ivanov on exactness of the loop functor and present the construction of a canonical specialization map from the $B^+_\dR$-affine Grassmannian to the Witt vector affine flag variety. 
\end{abstract}

\maketitle

\section{Introduction}
\label{section_introduction}

Let $E$ be a complete discretely valued field with ring of integers $\mathcal{O}_E$, and perfect residue field $k$ of characteristic $p$.
Set
\[
  A_E:=W(\mathcal{O}_C)\hat{\otimes}_{W(k)}\mathcal{O}_E
\]
for  $C$ some perfect non-archimedean field $C$ with ring of integers $\mathcal{O}_C$ such that $k\subseteq \mathcal{O}_C$. Here, $W(-)$ denotes the functor of ($p$-typical) Witt vectors.
If $E$ is of mixed characteristic, then $A_E$ is the period ring $A_\inf$ (associated with $C$ and $E$) which was considered by Fontaine. If $E\cong k[[\pi]]$ is of equal characteristic, then $A_E=W(\mathcal{O}_C)\hat{\otimes}_{W(k)}\mathcal{O}_E\cong \mathcal{O}_C[[\pi]]$ is a ring of power series over $\mathcal{O}_C$. 

Let $s\in \Spec(A_E)$ be the unique closed point and set
\[
  U_E:=\Spec(A_E)\setminus\{s\}.
\]
This ring $A_E$ has the decisive property that the restriction functor
\[
  \mathrm{Bun}(\Spec(A_E))\to \mathrm{Bun}(U_E)
\]
of vector bundles is an equivalence, cf.\ \Cref{lemma-extending-vector-bundles}.
Let $\mathcal{G}$ be a parahoric group scheme over $\mathcal{O}_E$, cf.\ \cite{bruhat_tits_groupes_reductif_sur_un_corps_local_II_schemas_en_groupes}, \cite{Landvogt1996}, \cite{tits_reductive_groups_over_local_fields}.
Our main theorem is the following generalization.

\begin{theorem}[cf.\ \Cref{sec:extend-tors-mathrmsp-1-main-theorem-a-e}]
  \label{theorem_introduction_main_theorem_a}
  All $\mathcal{G}$-torsors on $U_E$ extend to $\Spec(A_E)$.
\end{theorem}

If $C$ is algebraically closed, the extension is necessary a trivial $\mathcal{G}$-torsor because then $A_E$ is strictly henselian. In this case, \Cref{theorem_introduction_main_theorem_a} is equivalent to the statement that each $\mathcal{G}$-torsor on $U_E$ is trivial.

Let
\[
  R_E:=\mathcal{O}_E[[z]]
\]
be the formal power series ring over $\mathcal{O}_E$. Again vector bundles on its punctured spectrum extend to the whole spectrum. From \Cref{theorem_introduction_main_theorem_a} we can draw the following corollary by descent, cf.\ \Cref{proposition_criterion_for_extending_noetherian_case}.

\begin{corollary}[cf.\ \Cref{proposition_criterion_for_extending_noetherian_case}]
  \label{sec:introduction-corollary-main-theorem}
  All $\mathcal{G}$-torsors on the punctured spectrum of $\Spec(R_E)$ extend to $\Spec(R_E)$.
\end{corollary}

If $E$ has mixed characteristic \Cref{sec:introduction-corollary-main-theorem} has been treated in \cite{kisin_pappas_integral_models_of_shimura_varieties_with_parahoric_level_structure}, at least if the generic fiber $G:=\mathcal{G}\otimes_{\mathcal{O}_E}E$ splits over some tamely ramified extension of $E$, and is applied there to the construction of integral models of Shimura varieties.
If $E$ has equal characteristic both statements are used in \cite{hartl_viehmann_the_generic_fiber_of_moduli_spaces_of_bounded_local_g_shtukas} and work of Paul Breutmann, cf.\ \cite{breutmann2019functoriality}.
If $E$ has mixed characteristic \Cref{theorem_introduction_main_theorem_a} has applications to mixed characteristic affine Grassmannians and moduli spaces of mixed characteristic local shtukas (cf.\ \Cref{section_specialization_map}, \cite[Section 21.2.]{scholze2020berkeley}, \cite{gleason2020specialization}). These applications were our main motivation for considering the question of extending $\mathcal{G}$-torsors.

As the question is rather specific (and technical), let us give a more detailed motivation for considering it. The first motivation comes from the construction of integral models of Shimura varieties with parahoric level structure (cf.\ \cite{kisin_pappas_integral_models_of_shimura_varieties_with_parahoric_level_structure}, \cite{kisin_integral_models_for_shimura_varieties_of_abelian_type}, \cite{kisin_integral_canonical_models_of_shimura_varieties_summary}). A key step (cf.\ \cite[Key lemma 3.2.1.]{kisin_integral_canonical_models_of_shimura_varieties_summary}, \cite[Introduction]{kisin_pappas_integral_models_of_shimura_varieties_with_parahoric_level_structure}) is to prove (as the authors of \cite{kisin_pappas_integral_models_of_shimura_varieties_with_parahoric_level_structure} phrase it) ``that the crystalline realizations of certain Hodge cycles have good $p$-adic integrality properties''. For example, even in the case the level is hyperspecial, it is by no means clear that the crystalline realizations (obtained via de Rham/Betti/\'etale comparsion theorems and the theory of Breuil-Kisin modules) of the Hodge cycles defining the Shimura varieties define a \textit{parahoric} group scheme. But this necessary result is implied if torsors under parahoric group schemes over the punctured spectrum of $R_E$ extend with $\mathcal{O}_E$ in mixed characteristic, cf.\ \cite[Corollary 3.3.6.]{kisin_pappas_integral_models_of_shimura_varieties_with_parahoric_level_structure}.
Our main results will generalize the crucial \cite[Proposition 1.4.3]{kisin_pappas_integral_models_of_shimura_varieties_with_parahoric_level_structure} and it is therefore reasonable to expect that our results will have some applications to integral models of Shimura varieties (e.g.\ to integral models of unitary Shimura varieties in residue characteristic $2$).
Another line of motivation comes from the study of mixed-characteristic affine Grassmannians or affine flag varieties \cite{scholze2020berkeley}, especially \cite[Theorem 21.2.1]{scholze2020berkeley}. Here, the proof of ind-properness rests (in a subtle) way on \Cref{theorem_introduction_main_theorem_a}. In a similar vein, \Cref{theorem_introduction_main_theorem_a} is applied in the calculation of connected components of moduli spaces of mixed characteristic local shtukas, \cite{gleason2020specialization}.

Let us describe our strategy for proving \Cref{theorem_introduction_main_theorem_a}, and where it departs from the arguments of \cite[Proposition 1.4.3.]{kisin_pappas_integral_models_of_shimura_varieties_with_parahoric_level_structure}. Using descent we may enlarge the field $C$ as we like and assume that it is maximally complete, cf.\ \Cref{lemma_descent_along_extension_of_c}. Let $\pi$ be a uniformizer in $\mathcal{O}_E$. We prove that \Cref{theorem_introduction_main_theorem_a} is equivalent to the statement that each $G$-torsor on
\[
  V:=U[1/\pi]=\Spec(A_E[1/\pi])
\]
is trivial. This uses crucially that $\mathcal{G}$ is parahoric, cf.\ \Cref{coro:criterion-extending-torsors-parahoric}. The analogous reduction is not possible for the ring $R_E$ because we are applying Steinberg's theorem to the field
\[
  (A_{E,(\pi)})^\wedge_\pi[1/\pi],
\]
a finite extension of $W(C)[1/p]$ resp.\ $C((\pi))$, which is of cohomological dimension $1$. Using a result of Gabber on Brauer groups on affine schemes, we prove then the case of tori, cf.\ \Cref{proposition_cohomology_of_tori}. This allows us to assume that $G$ is quasi-split, semisimple, simply connected, and to reduce to proving that the map
\[
  H^1(V,G)\to H^1(U_{\crys},G)
\] has trivial kernel, where
\[
  U_{\crys}\subseteq \Spec(A_E)
\] is the complicated ``crystalline'' part, cf.\ \Cref{lemma_spectrum_of_a_inf}. We may even assume that $G$ has no factors of type $E_8$ because quasi-split groups of type $E_8$ are split, and the split case is easy, cf.\ \Cref{corollary_h1_trivial_for_split_groups}. We then introduce (following a suggestion of Scholze) a ``decompletion'' $\tilde{A}_E$ of $A_E$ using Mal'cev-Neumann series over $\mathcal{O}_E$, cf.\ \Cref{sec:decompleting-a_e}, which is roughly a power series ring over $\mathcal{O}_E$ with exponents in $\R$. The ring $\tilde{A}_E$ is henselian along $\pi$ and thus by a result of Gabber/Ramero $G$-torsors on $V$ can be ``decompleted'' to $G$-torsors on
\[
  \tilde{V}:=\Spec(\tilde{A}_E[1/\pi]).
\]
Crucially, the ``crystalline part'' of $\tilde{A}_E[1/\pi]$ is simpler, namely a valuation ring of rank $1$, cf.\ \Cref{sec:decompleting-a_e-4-crystalline-part-of-tilde-a-e}. Using that reductions to Borels extend along valuation rings, we reduce to checking triviality at the generic point of $\Spec(\tilde{A}_E[1/\pi])$. For this, we use Gille's result on Serre's conjecture II, cf.\ \cite{gille_cohomologie_galoisienne_des_groupes_quasi_deployes_sur_des_corps_de_dimension_cohomologique_2}, i.e., that torsors under quasi-split, semisimple, simply connected groups without $E_8$-factors over fields of dimension $\leq 2$ are trivial. The assumptions on $G$ are met by our previous reduction, and thus we have to show that the fraction field $\tilde{K}_E$ of $\tilde{A}_E$ is of cohomological dimension $\leq 2$ (plus an additional bound on its Kato cohomology in equal characteristic, cf.\ \Cref{sec:decompleting-a_e-1-kato-cohomologie-of-tilde-k-e}). We do this in \Cref{sec:decompleting-a_e-1-cohomological-dimension-of-tilde-a-e} by using perfectoid methods and cohomological bounds on the $\F_p$-cohomology of perfectoid rings, cf.\ \cite[Theorem 11.1.]{bhatt_scholze_prisms_and_prismatic_cohomology}, and \'etale cohomology of diamonds, cf.\ \cite[Proposition 21.11.]{scholze_etale_cohomology_of_diamonds}. Gille's result is also used in the proof of \cite[Proposition 1.4.3.]{kisin_pappas_integral_models_of_shimura_varieties_with_parahoric_level_structure}, and the needed cohomological bounds are easier to prove there. Our proof is simpler in the respect that we can avoid subtleties with parahoric group schemes. We want to mention that Jo\~{a}o Louren\c{c}o has informed us that he could improve the needed results on parahoric group schemes and thus obtain \Cref{sec:introduction-corollary-main-theorem} via the strategy of \cite[Proposition 1.4.3.]{kisin_pappas_integral_models_of_shimura_varieties_with_parahoric_level_structure}, cf.\ remark after \cite[Lemme 2.6]{lourencco2019grassmanniennes}.

We give two applications of our main result. First, we prove that for every perfect $k$-algebra $R$ each $G$-torsors on
\[
  A_E(R)[1/\pi]:=W(R)\hat{\otimes}_{W(k)} E
\]
can be trivialized $v$-locally, cf.\ \Cref{sec:v-stack-g-2-torsors-trivial-in-v-topology}. This uses results of Ivanov, cf.\ \cite{ivanov2020ind}, on arc-descent for finite projective $A_E(R)[1/\pi]$-modules, and allows us to extend (a reformulation of) \Cref{theorem_introduction_main_theorem_a} to open and bounded valuation rings $C^+\subseteq C$, cf.\ \Cref{sec:v-local-triviality-complete-valuation-ring}. 
As a second application of our result we present in \Cref{section_specialization_map} the construction of a specialization map between mixed characteristic affine Grassmannianns. 

\subsection*{Acknowledgements} We thank Peter Scholze heartily for his interest and suggestions regarding this topic, especially for his suggestion to consider Mal'cev-Neumann series. Moreover, we heartily thank K\k{e}stutis \v{C}esnavi\v{c}ius, Ian Gleason, Daniel Kirch, Jo\~{a}o Louren\c{c}o and Sebastian Posur for discussions and comments on this paper.

\section{Notations}
\label{sec:notations}

We fix some notation used throughout the paper.
Let $E$ be a complete discretely valued field with ring of integers $\mathcal{O}_E$, perfect residue field $k$ of characteristic $p>0$ and let $\pi\in \mathcal{O}_E$ be a uniformizer. 
We denote by $\mathcal{G}$ a parahoric group scheme over $\mathcal{O}_E$ (cf.\ \cite{bruhat_tits_groupes_reductif_sur_un_corps_local_II_schemas_en_groupes}, \cite{tits_reductive_groups_over_local_fields}, \cite{landvogt_a_compactification_of_the_bruhat_tits_building}) and set 
$$
G:=\mathcal{G}\otimes_{\mathcal{O}_E}E
$$ 
for its (connected) reductive fiber over $E$.
Let $C$ be a perfect, complete non-archimedean extension such that $k\subseteq \mathcal{O}_C$ where $\mathcal{O}_{C}$ is the ring of integers of $C$. Let $\mathfrak{m}_{C}\subseteq \mathcal{O}_{C}$ be the maximal ideal of $\mathcal{O}_C$ and let $k^\prime:=\mathcal{O}_{C}/\mathfrak{m}_{C}$ be the residue field of $\mathcal{O}_{C}$. For simplicity we will mostly (without loosing generality) assume that $k\cong k^\prime$.
For a perfect ring $S$ we denote its ring of ($p$-typical) Witt vectors by $W(S)$.
Set
$$
A_{E}:=W(\mathcal{O}_{C})\hat{\otimes}_{W(k)}\mathcal{O}_E.
$$
Then
$$
A_E\cong 
\begin{cases}
  \mathcal{O}_C[[\pi]] & \textrm{ if } \mathrm{char}(E)>0 \\
  A_{\inf}\otimes_{W(k)}\mathcal{O}_E & \textrm{ if } \mathrm{char}(E)=0
\end{cases}
$$
with $A_{\inf}=W(\mathcal{O}_C)$ Fontaine's ring associated with $C$. We may occasionally drop the $E$ from the notation.
Let
$$
[\cdot]\colon \mathcal{O}_C\to A_E
$$
be the Teichm\"uller lift. Then every element $a\in A_E$ can be uniquely written as
$$
a=\sum\limits_{i=0}^\infty [a_i]\pi^i
$$
with $a_i\in \mathcal{O}_C$.
Let
$$
s:=s_{A_E}\in \Spec(A_E)
$$
be the unique closed point given by the unique maximal ideal
$$
\mathfrak{m}:=\{\ \sum\limits_{n\geq 0}^\infty [a_n]\pi^n\ |\ a_0\in \mathfrak{m}_{C}\}
$$
of $A_E$.
Let
$$
U:=U_{E}:=\Spec(A_E)\setminus\{s\}
$$
be the punctured spectrum of $A_E$. Moreover, set
$$
V:=V_{E}:=\Spec(A_E[1/\pi])\subseteq U.
$$
Finally, define the ``crystalline point''
$$
\mathfrak{p}_{\crys}:=\{\sum\limits_{n\geq 0}^\infty [a_n]\pi^n\ |\ a_n\in \mathfrak{m}_{C} \textrm{ for all }n\}
$$
and the ``crystalline part''
$$
U_\crys:=\Spec(A_{E,\mathfrak{p}_\crys})\subseteq U.
$$
For a finite field extension $E^\prime/E$ we denote by $\mathcal{O}_{E^\prime}$ its ring of integers. Note that
$$
A_{E^\prime}\cong A_{E}\otimes_{\mathcal{O}_E}\mathcal{O}_{E^\prime}
$$
and
$$
U_{E^\prime}=U_{E}\otimes_{\mathcal{O}_E}\mathcal{O}_{E^\prime}, V_{E^\prime}=V_{E}\otimes_E E^\prime, \textrm{ etc.}
$$

If not stated explicitly otherwise, $H^\ast$ will always mean \'etale cohomology. 

\section{The spectrum of $A_E$}
\label{section_spectrum_of_a}

We use the notation from \Cref{sec:notations}.
Moreover, we assume that $C$ is algebraically closed.
Recall that an element
$$
\xi\in A_{E}
$$
is called distinguished (or primitive) of degree $1$ (cf.\ \cite[D\'efinition 2.2.1.]{fargues_fontaine_courbes_et_fibres_vectoriels_en_theorie_de_hodge_p_adique}) if
$$
\xi=u(\pi-[\varpi])
$$
for some unit $u\in A_{E}^\times$ and some $\varpi\in \mathfrak{m}_C$. 

For $E$ of equal characteristic the next lemma can be found in \cite[Lemma 8.3]{hartl_viehmann_the_generic_fiber_of_moduli_spaces_of_bounded_local_g_shtukas} as well.

\begin{lemma}
\label{lemma_spectrum_of_a_inf}
The spectrum $\Spec(A_{E})$ of $A_{E}$ is given as
$$
\Spec(A_{E})=U_\crys\cup \{\mathfrak{m}\} \cup \bigcup\limits_{\xi\in {A_E} \atop \textrm{distinguished of degree } 1 } \{(\xi)\}.
$$
\end{lemma}
\begin{proof}
Let $\mathfrak{p}\subseteq A_E$ be an arbitrary prime ideal. If $\mathfrak{p}$ contains a distinguished element $\xi$ of degree $1$ (or equivalently some power), then $\mathfrak{p}$ lies in the subset
$$
\Spec(A_E/(\xi))\subseteq \Spec(A_E).
$$
But $(\xi)$ being distinguished of degree $1$ implies that $A_E/(\xi)$ is isomorphic to the ring of integers $\mathcal{O}_{C^\#}$ for some non-archimedean field $C^\#$, possibly $C^\#\cong C$ (cf.\ \cite[Corollaire 2.2.23]{fargues_fontaine_courbes_et_fibres_vectoriels_en_theorie_de_hodge_p_adique}). The ring $\mathcal{O}_{C^\#}$ contains exactly two prime ideals, namely $(0)$ and $\mathfrak{m}/(\xi)$. In particular, $\mathfrak{p}=(\xi)$ or $\mathfrak{p}=\mathfrak{m}$. Now assume that $\mathfrak{p}$ does not contain a distinguished element $\xi\in A_E$. We want to prove that 
$$
\mathfrak{p}\subseteq \mathfrak{p}_{\mathrm{cris}}=\{\sum\limits_{n\geq 0}[x_i]\pi^i\ |\ x_i\in \mathfrak{m}_{C}\ \}.
$$ Assume the contrary. Then
$$
0\neq(\mathfrak{p}+\mathfrak{p}_{\mathrm{cris}})/\mathfrak{p}_{\mathrm{cris}}\subseteq A_E/\mathfrak{p}_{\mathrm{cris}}
$$
is non-zero and there exists an element $a\in \mathfrak{p}$ not mapping to zero in $A_E/\mathfrak{p}_{\mathrm{cris}}$.
Write
$$
a=\sum\limits_{i=0}^\infty [x_i]\pi^i
$$ 
with $x_i\in \mathcal{O}_{C}$. As $\pi\notin \mathfrak{p}$ and $\mathfrak{p}$ is prime we can assume $x_0\neq 0$ after dividing possibly by some power of $\pi$. Moreover, one $x_i$ must be a unit in $\mathcal{O}_{C}$ as $a$ does not map to $0$ in $A_E/\mathfrak{p}_{\mathrm{cris}}$. In other words, $a$ is primitive in the sense of \cite[D\'efinition 2.2.1.]{fargues_fontaine_courbes_et_fibres_vectoriels_en_theorie_de_hodge_p_adique}. By \cite[Th\'eor\`eme 2.4.1.]{fargues_fontaine_courbes_et_fibres_vectoriels_en_theorie_de_hodge_p_adique} resp.\ \cite{lazard_les_zeros_des_fonctions_analytiques_dune_variable_sur_un_corps_value_complet} the element $a$ can be written as a product
$$
a=\prod_{i=1}^n a_i
$$
for some distinguished elements $a_i$ of degree $1$ as we assumed that $C$ is algebraically closed. As $\mathfrak{p}$ is a prime ideal, one of these $a_i$ must lie in $\mathfrak{p}$ which is the contradiction we were looking for. This finishes the proof.
\end{proof}

We can record the following corollary.

\begin{corollary}
\label{corollary_local_rings_at_distinguished_elements_are_dvr}
Let $\xi\in A_E$ be a distinguished element of degree $1$. Then the local ring
$$
A_{E,(\xi)}
$$
is a discrete valuation ring.
\end{corollary}
\begin{proof}
Without knowing that the local ring $A_{E,(\xi)}$ is a discrete valuation ring it is known (by \cite[D\'efinition 2.7.1.]{fargues_fontaine_courbes_et_fibres_vectoriels_en_theorie_de_hodge_p_adique}) that the $\xi$-adic completion of $A_{E,(\xi)}$ is a discrete valuation ring (at least for $(\xi)\neq (\pi)$, but the remaining case $(\pi)=(\xi)$ is clear).
Let $\mathfrak{p}\subseteq A_E$ be a prime ideal contained in $(\xi)$, i.e., $\mathfrak{p}$ lies in the spectrum
$$
\Spec(A_{E,(\xi)})\subseteq \Spec(A_E)
$$
of the localisation $A_{E,(\xi)}$.
Every $a\in \mathfrak{p}$ can be written as
$$
a=b\xi
$$
for some $b\in A_E$. If we assume $\xi\notin \mathfrak{p}$, then
$$
b=\frac{a}{\xi}\in \mathfrak{p},
$$
i.e., $\xi\mathfrak{p}=\mathfrak{p}$. But $A_E$ injects into the $\xi$-adic completion
$$
R:=(A_{E,(\xi)})^\wedge_\xi,
$$ which is a discrete valuation ring with uniformizer $\xi\in R$.
Let 
$$
\mathfrak{q}=\mathfrak{p}R.
$$
Then $\xi\mathfrak{q}=\mathfrak{q}$ which implies
$$
\mathfrak{q}=0
$$
as $R$ is a discrete valuation ring. But then $\mathfrak{p}=0$ as well.
In other words, we have proven that the spectrum
$$
\Spec(A_{E,(\xi)})=\{(\xi),(0)\}
$$
contains exactly two prime ideals, both of which are finitely generated. By \cite[Chapitre 0, Proposition (6.4.7.)]{grothendieck_dieudonne_ega_I_second_edition} this implies that $A_{E,(\xi)}$ is noetherian and then more precisely a discrete valuation ring.
\end{proof}

We remark that the subset 
$$
U_\crys\subseteq \Spec(A_E)
$$
remains mysterious. For example it contains the non-closed prime ideal
$$
\bigcup\limits_{\varpi\in \mathfrak{m}_C}[\varpi]A_E\subsetneq \mathfrak{p}_{\mathrm{cris}}
$$ (cf.\ \cite[Section 1.10.4.]{fargues_fontaine_courbes_et_fibres_vectoriels_en_theorie_de_hodge_p_adique}).
In particular, the Krull dimension of $A_E$ is at least $3$.
In fact, the Krull dimension of $A_E$ is infinite, cf.\ \cite{lang2019ainf}.
One main point in our proof will be to pass to a ``decompletion'' $\tilde{A}_E$ of $A_E$ for which the crystalline part has an easier structure, it is a valuation ring of rank $1$, cf.\ \Cref{sec:decompleting-a_e-4-crystalline-part-of-tilde-a-e}.

\section{Some commutative algebra over $A_E$}

We shortly want to mention some results on commutative algebra over $A_E$ generalizing those in \cite[Chapter 4]{bhatt_morrow_scholze_integral_p_adic_hodge_theory}.
Recall that $\pi\in \mathcal{O}_E$ is a uniformizer, that $s\in \Spec(A_E)$ denotes the unique closed point of $\Spec(A_E)$ and that $U=\Spec(A_E)\setminus\{s\}$ is the punctured spectrum of $A_E$.  
The proof of \cite[Lemma 4.6.]{bhatt_morrow_scholze_integral_p_adic_hodge_theory} generalizes to the following lemma.

\begin{lemma}
\label{lemma-extending-vector-bundles}
The restriction of vector bundles induces an equivalence of categories between vector bundles on $\Spec(A_E)$ and vector bundles on $U$. In particular, all vector bundles on $U$ are free.   
\end{lemma}
\begin{proof}
Replacing $p$ by $\pi$ in \cite[Lemma 4.6.]{bhatt_morrow_scholze_integral_p_adic_hodge_theory} the same proof works and we refer the reader to its proof.
\end{proof}

The next corollary is (nearly) \cite[Corollary 4.12.]{bhatt_morrow_scholze_integral_p_adic_hodge_theory}.

\begin{corollary}
\label{propositon-locally-free-sheaves-after-inverting-p}
Let $N$ be a finite projective $A_E[1/\pi]$-module. Then $N$ is free.  
\end{corollary}
\begin{proof}
Let $M\subset N$ be a finitely generated $A_E$-submodule such that $M[1/\pi]=N$. The localisation
$$
A_{E,(\pi)}
$$ 
of $A_E$ at the prime ideal $(\pi)$ for $\pi\in \mathcal{O}_E$ a uniformizer is a discrete valuation ring (cf.\ \Cref{corollary_local_rings_at_distinguished_elements_are_dvr} or \cite[Lemma 4.10]{bhatt_morrow_scholze_integral_p_adic_hodge_theory}). As $M$ is finitely generated and $\pi$-torsion free the localized module $M\otimes_{A_E}A_{E,(\pi)}$ is finite free. Using Beauville-Laszlo (cf.\ \cite{beauville_laszlo_un_lemme_de_descente}) (and that $\Spec(A_E/p)\cong \Spec(\mathcal{O}_C)$ has exactly two points) the quasi-coherent sheaf $\mathcal{M}$ on $\Spec(A_E)$ defined by $M$ restricts thus to a vector bundle on the punctured spectrum $U$. By \Cref{lemma-extending-vector-bundles} this vector bundle is trivial which implies that $N$ is already free.
\end{proof}

In particular, we can conclude that every line bundle on $U$ resp.\ $V$ is trivial, i.e.,
$$
H^1(U,\mathbb{G}_m)=H^1(V,\mathbb{G}_m)=\{1\}.
$$



\section{A general affinoid comparison theorem}
\label{sec:gener-affin-comp}

Later we want to prove, if $C$ is algebraically closed and maximally complete, a cohomological bound for the \'etale cohomology of
\[
  V=\Spec(A_E[1/\pi])
\]
by deducing it from results on the \'etale cohomology of the
associated adic space
\[
  \Spa(A_E[1/\pi],A_E)
\]
(where $A_E$ carries the $\pi$-adic, and not the $(\pi,[\varpi])$-adic topology). For this we need to slightly improve a result of Huber resp.\ Hansen, cf.\ \cite[Corollary 3.2.3.]{Huber1996a} resp.\ \cite[Theorem 1.9.]{hansen2020vanishing}, on comparing \'etale cohomology of schemes and adic spaces/diamonds in the affine case.

Let $(A,A^+)$ be a pair of a Tate\footnote{I.e., a complete topological $\Z_p$-algebra $A$, containing a topologically nilpotent unit and an open and bounded subring.} $\Z_p$-algebra $A$ with an open and integrally closed subring $A^+$. To this pair one can associate the spatial diamond
\[
  \mathrm{Spd}(A,A^+)
\]
sending a perfectoid space $X$ over $\F_p$ to the set of triples
\[
  (X^\sharp, \iota, f),
\]
where $X^\sharp$ is a perfectoid space (over $\Z_p$), $\iota\colon (X^\sharp)^\flat\cong X$ an identification of the tilt of $X^\sharp$ with $X$, and $f\colon A\to \mathcal{O}_{X^\sharp}(X^\sharp)$ a continuous $\Z_p$-algebra homomorphism sending $A^+$ to $\mathcal{O}_{X^\sharp}^+(X^\sharp)$, cf.\ \cite[Proposition 15.4.]{scholze_etale_cohomology_of_diamonds}.

Let us note that there exists a natural morphism
\[
  \mu\colon \Spd(A,A^+)_{\et}\to \Spec(A)_{\et}
\]
of \'etale sites (here $\Spd(A,A^+)_{\et}$ denotes the site from \cite[Definition 14.1.]{scholze_etale_cohomology_of_diamonds}).

Combined with \cite[Lemma 15.6.]{scholze_etale_cohomology_of_diamonds} the following result was obtained under noetherianess assumptions by Huber/Hansen. The general case follows by an easy noetherian approximation, cf.\ \cite[the proof of Theorem 4.10]{vcesnavivcius2019purity} for a similar argument. 

\begin{proposition}
  \label{sec:gener-affin-comp-1-affinoid-comparison-theorem}
  Let $A$ be a Tate $\Z_p$-algebra and $A^+\subseteq A$ an open and integrally closed subring.  Then for each \'etale torsion abelian sheaf $\mathcal{F}$ on $\Spec(A)$ and any $i\geq 0$ the canonical morphism
  \[
    H^i(\Spec(A)_{\et},\mathcal{F})\to H^i(\Spd(A,A^+)_\et,\mu^\ast \mathcal{F})
  \]
  is an isomorphism.
\end{proposition}
\begin{proof}
  Both sides commute with colimits in $\mathcal{F}$ (as the sites $\Spec(A)_{\et}$ and $\Spd(A,A^+)_{\et}$ are coherent, the latter by the fact that $\Spd(A,A^+)$ is a \textit{spatial} diamond). Thus we may assume that $\mathcal{F}$ is constructible.
  Fix a pseudo-uniformizer $t\in A^+$.
  Recall that the topology on every open bounded subring $A_0$ of $A^+$ is $t$-adic and $A=A_0[1/t]$, cf.\ \cite[Chapter 2]{scholze2020berkeley}, and that $A^+$ is the filtered union of its open and bounded subrings (which wlog contain $t$).
  We can therefore write
  \[
    A^+=\varinjlim\limits_{i\in I} B_i
  \]
  as a filtered colimit of $\Z_p[t]$-algebras, which are the $t$-adic completions of $\Z_p$-algebras of finite type.
  Then $\mathcal{F}$ is defined already over $\Spec(B_{i_0}[1/t])$ for some $i_0\in I$, i.e., it is the pullback of some constructible \'etale sheaf $\mathcal{F}_{i_0}$ on $\Spec(B_{i_0}[1/t])$. We may assume that $i_0\in I$ is initial. For $i\in I$ set
  \[
    \mathcal{F}_i
  \]
  as the pullback of $\mathcal{F}_{i_0}$ along $\Spec(B_i[1/t])\to \Spec(B_{i_0}[1/t])$. Then
  \[
    H^\ast(\Spec(A)_{\et},\mathcal{F})\cong \varinjlim\limits_{i\in I} H^\ast(\Spec(B_i[1/t])_{\et},\mathcal{F}_i)
  \]
  by general results on \'etale cohomology of schemes. On the other hand, note that by definition
  \[
    \Spd(A,A^+)=\varprojlim\limits_{i\in I^{\mathrm{op}}} \Spd(B_i[1/t],B_i),
  \]
  and thus by \cite[Proposition 14.9.]{scholze_etale_cohomology_of_diamonds}
  \[
    H^\ast(\Spd(A,A^+)_{\et},\mathcal{F})\cong \varinjlim\limits_{i\in I} H^\ast(\Spd(B_i[1/t],B_i)_{\et},\mathcal{F}_i)
  \]
  Therefore we may replace $A$ by some $B_i[1/t]$ and $B_i$ by its (noetherian) integral closure in $B_i[1/t]$, and apply \cite[Theorem 1.9.]{hansen2020vanishing} to conclude.
\end{proof}

\section{Decompleting $A_E$}
\label{sec:decompleting-a_e}

In this section we want to introduce (if $C$ is algebraically closed and maximally complete) a certain ``decompletion''
\[
  \tilde{A}_E\subseteq A_E
\]
of $A_E$, which has the crucial property that its ``crystalline'' part is much easier, namely the spectrum of a valuation ring of rank $1$, cf.\ \Cref{sec:decompleting-a_e-4-crystalline-part-of-tilde-a-e}. Moreover, we will prove that the fraction field of $\tilde{A}_E$ has cohomological dimension $\leq 2$, cf.\ \Cref{sec:decompleting-a_e-1-cohomological-dimension-of-tilde-a-e}, \Cref{sec:decompleting-a_e-1-kato-cohomologie-of-tilde-k-e}, which is necessary as we want to apply Gille's result on Serre's conjecture II, cf.\ \cite{Gille2001} to it.

For the whole section we assume that the field $C$ is algebraically closed and maximally complete (in the sense of \cite{poonen1993maximally}). For convenience we assume that the value group of $C$ is $\R$. Let us note that by \cite[Corollary 6]{poonen1993maximally} any complete non-archimedean field can be embedded in an extension satisfying these conditions. In fact, one can construct such an extension explicitly using Mal'cev-Neumann series, cf.\ \cite{poonen1993maximally}.\footnote{We thank Peter Scholze for giving the hint to consider these rings.}

\begin{definition}
  \label{sec:decompleting-a_e-1-definition-malcev-von-neumann-rings}
  Let $R$ be any ring. Then the Mal'cev-Neumann ring $R((t^\R))$ with residue ring $R$ (and value group\footnote{One can replace $\R$ by any totally ordered abelian group in the definition.} $\R$) is the ring of formal power series
  \[
    a=\sum\limits_{x\in \R} a_xt^x
  \]
  with coefficients $a_x\in R$, such that the support
  \[
    \mathrm{supp}(a):=\{x\in \R \ |\ a_x\neq 0\}
  \]
  is a well-ordered subset of $\R$.\footnote{We recall that a totally ordered set is well-ordered if any non-empty subset has a minimal element.}
\end{definition}

By the condition on the support the $R$-module $R((t^\R))$ has the natural well-defined multiplication
\[
  (\sum\limits_{x\in \R}a_xt^x)(\sum\limits_{x\in \R} b_xt^x):=\sum\limits_{x\in \R}(\sum\limits_{y+z=x}a_yb_z)t^x
\]
turning it into a ring, cf.\ \cite[Section 3]{poonen1993maximally}.
Let us note that there is the natural valuation
\[
  \nu\colon R((t^\R))\to \R\cup \{\infty\},\ a\mapsto \mathrm{inf}\{\mathrm{supp}(a)\}
\]
on
\[
  R((t^\R)).
\]
Let
\[
  R[[t^{\R_{\geq 0}}]]:=\nu^{-1}(\R_{\geq 0}\cup \{\infty\}).
\]
Clealy, in the topology on $R((t^\R))$ induced by $\nu$ the subring $R[[t^{\R_{\geq 0}}]]$ is open and its subspace topology is the $t$-adic topology. Moreover,
\[
  R((t^{\R}))\cong R[[t^{\R_{\geq 0}}]][1/t].
\]

Let us turn back to the question of ``decompleting'' $A_E$.
Under the assumptions on $C$ from the beginning of the section we get that
\[
  C\cong k((t^\R))
\]
by \cite[Corollary 6]{poonen1993maximally}, where $k$ is the (algebraically closed) residue field of $C$. Of course, this isomorphism can be chosen to be compatible with valuations. We therefore redefine $C$ in this section by setting
\[
  C=k((t^\R)).
\]
In particular, we get
\[
  \mathcal{O}_C= k[[t^{\R_{\geq 0}}]]:=\nu^{-1}(\R_{\geq 0}\cap \{\infty\}).
\]
Recall that
\[
  A_E=W(\mathcal{O}_C)\hat{\otimes}_{W(k)}\mathcal{O}_E,
\]
where the completion is $\pi$-adic, cf.\ \Cref{sec:notations}.
We can now define our promised ``decompletion''
\[
  \tilde{A}_E:=\mathcal{O}_E[[t^{\R_{\geq 0}}]].
\]
Let us note that there exists a natural morphism
\[
  \tilde{A}_E\to A_E.
\]
of $\mathcal{O}_E$-algebras.
In fact, we have the following.

\begin{lemma}
  \label{sec:decompleting-a_e-3-pi-completion-of-tilde-a-e}
  The ring $\tilde{A}_E$ is $\pi$-torsion free and its $\pi$-adic completion identifies with $A_E$. Moreover, the ring $\tilde{A}_E$ is henselian along $(\pi)$.
\end{lemma}
\begin{proof}
  The $\pi$-torsion freeness is clear as $\tilde{A}_E$ embeds (as an $\mathcal{O}_E$-module) into the product $\prod\limits_{\R} \mathcal{O}_E$. Moreover,
  \[
    \tilde{A}_E/\pi\cong k[[t^{\R_{\geq 0}}]]=\mathcal{O}_C
  \]
  is perfect. This implies that
  \[
    (\tilde{A}_E)^\wedge_\pi\cong A_E
  \]
  as perfect $k=\mathcal{O}_E/\pi$-algebras lift uniquely to $\pi$-torsion free $\pi$-complete $\mathcal{O}_E$-algebras.
  To see that $\tilde{A}_E$ is henselian along $\pi$ we may (by $t$-completeness) argue mod $t$. The ideal
  \[
    \nu^{-1}(\R_{>0}\cup \{\infty\})/(t)\subseteq \tilde{A}_E/(t)
  \]
  is locally nilpotent and thus 
  \[
    \tilde{A}_E/t
  \]
  is henselian along $\pi$ if and only if
  \[
    \tilde{A}_E/\nu^{-1}(\R_{>0}\cup \{\infty\})\cong \mathcal{O}_E
  \]
  is henselian along $\pi$. But the latter holds as $\mathcal{O}_E$ is $\pi$-adically complete.
\end{proof}

This has the following consequence.

\begin{lemma}
  \label{sec:decompleting-a_e-5-torsors-on-tilde-a-e-1-over-pi}
  Let $H$ be any smooth affine group scheme over $\mathcal{O}_E$. Then the canonical map
  \[
    H^1_{\et}(\Spec(\tilde{A}_E[1/\pi]),H)\to H^1_{\et}(\Spec(A_E[1/\pi]),H)
  \]
  is a bijection.
\end{lemma}
\begin{proof}
  This follows from \Cref{sec:decompleting-a_e-3-pi-completion-of-tilde-a-e}, \cite[Theorem 5.8.14.]{gabber2003almost} and the observation that by smoothness any $H$-torsor for the fpqc-topology admits sections \'etale locally.
\end{proof}

Set
\[
  \tilde{\mathfrak{p}}_{\crys}:=\nu^{-1}(\R_{>0}\cup \{\infty\})
\]
and define the ``crystalline'' part of $\Spec(\tilde{A}_E)$ as
\[
  \tilde{U}_\crys:=\Spec(\tilde{A}_{E,\tilde{\mathfrak{p}}_\crys}).
\]

The following observations are crucial.

\begin{lemma}
  \label{sec:decompleting-a_e-4-crystalline-part-of-tilde-a-e}
  The following hold true:
  \begin{enumerate}
  \item[1)] Under $\Spec(A_E)\to \Spec(\tilde{A}_E)$ the prime ideal $\mathfrak{p}_\crys\subseteq A_E$ maps to $\tilde{\mathfrak{p}}_{\crys}\subseteq \tilde{A}_E$. In particular, $U_\crys\subseteq \tilde{U}_\crys\times_{\Spec(\tilde{A}_E)} \Spec(A_E)$.
  \item[2)] The ring
    \[
      \tilde{A}_{E,\tilde{\mathfrak{p}}_\crys}
    \]
    is a valuation ring (of rank $1$).
  \end{enumerate}
\end{lemma}

Thus, by replacing $A_E$ by $\tilde{A}_E$ the structure of the crystalline part has become much simpler.

\begin{proof}
  For $1)$: Note that $\mathfrak{p}_\crys$ resp.\ $\tilde{\mathfrak{p}}_\crys$ are the images of the generic point of the canonical morphism
  \[
    \Spec(\mathcal{O}_E)\to \Spec(A_E)
  \]
  resp.\
  \[
    \Spec(\mathcal{O}_E)\to \Spec(\tilde{A}_E).
  \]
  Now it suffices to see that the morphism $\tilde{A_E}\to \mathcal{O}_E$ factors over $A_E$.

  For $2)$: The localization $\tilde{A}_{E,\tilde{\mathfrak{p}}_\crys}$ is the valuation ring $R\subseteq \mathrm{Frac}(\tilde{A}_E)$ induced by the valuation
  \[
    \nu\colon \tilde{A}_E\to \R\cup \{\infty\}.
  \]
  Indeed, the assumptions of \Cref{sec:decompleting-a_e-1-criterion-for-localization-to-be-a-valuation-ring} are satisfied and we can conclude.
\end{proof}

\begin{lemma}
  \label{sec:decompleting-a_e-1-criterion-for-localization-to-be-a-valuation-ring} Let $R$ be an integral domain, and let $\omega\colon Q:=\mathrm{Frac}(R)\to \Gamma\cup \{\infty\}$ be an (arbitrary) valuation on its fraction field, such that $\omega(R)\subseteq \Gamma_{\geq 0}\cup \{\infty\}$.
  Set
  \[
    \mathfrak{p}:=\omega^{-1}(\Gamma_{>0}\cup \{\infty\}).
  \]
  Then $\mathfrak{p}\subseteq R$ is a prime ideal in $R$ and if $\omega^{-1}(\Gamma_{\geq \gamma}\cup \{\infty\})$ is a principal ideal for any $\gamma\in \Gamma$, the localization $R_{\mathfrak{p}}$ is a valuation ring (with valuation $\omega$).
\end{lemma}
\begin{proof}
  It is clear that $\mathfrak{p}$ is a prime ideal in $R$.
  Let $r,s\in R\setminus\{0\}$. Set
  \[
    \gamma_1:=\omega(r),\ \gamma_2:=\omega(s).
  \]
  Without losing generality we may assume $\gamma_1\geq \gamma_2$. For $i=1,2$ let
  \[
    t_i\in R
  \]
  be a generator of $\omega^{-1}(\Gamma_{\geq \gamma_i}\cup \{\infty\})$. Using our assumption that $\omega(R\setminus \{0\})\subseteq \Gamma_{\geq 0}$ we note that
  \[
    \omega(t_i)=\gamma_i.
  \]
  Then $(t_1)\subseteq (t_2)$ and there exists $a,b,c\in R$ with
  \[
    at_2=t_1,\ bt_1=r,\ ct_2=s.
  \]
  Thus,
  \[
    \frac{r}{s}=\frac{bt_1}{ct_2}=\frac{a b t_2}{ct_2}=\frac{ab}{c}.
  \]
  As
  \[
    \omega(t_2)=\gamma_2=\omega(s),
  \]
  we can conclude
  \[
    \omega(c)=0,
  \]
  i.e., $c\notin \mathfrak{p}$. This finishes the proof.
\end{proof}

Our next aim is to prove that the fraction field
\[
 \tilde{K}_E:=\mathrm{Frac}(\tilde{A}_E)
\]
of $\tilde{A}_E$ has $\ell$-cohomological dimension $\mathrm{cd}_{\ell}(\tilde{K}_E)\leq 2$ for all primes $\ell$.

\begin{theorem}
  \label{sec:decompleting-a_e-1-cohomological-dimension-of-tilde-a-e}
  Let $\ell$ be any prime. Then
  \[
    \mathrm{cd}_\ell(\tilde{K}_E)\leq 2.
  \]
\end{theorem}
\begin{proof}
  The case $\ell=p=\mathrm{char}(E)$ follows by Artin-Schreier theory, hence we may assume that we are outside this case, i.e., $\ell\neq p$ if $\mathrm{char}(E)=p$.
  Let
  \[
    j\colon \Spec(\tilde{K}_E)\to \tilde{V}:=\Spec(\tilde{A}_E[1/\pi])
  \]
  be the natural inclusion, and let $\mathcal{F}$ be an \'etale $\F_\ell$-sheaf on $\Spec(\tilde{K}_E)$. Then we obtain a distinguished triangle
  \[
    j_\ast \mathcal{F} \to Rj_\ast \mathcal{F}\to Q
  \]
  in $D(\tilde{V}_\et,\F_\ell)$.
  Pulling back along the morphism
  \[
    f\colon V:=\Spec(A_E[1/\pi])\to \tilde{V}
    \]
    yields the triangle
    \[
      f^\ast j_\ast\mathcal{F}\to f^\ast Rj_\ast \mathcal{F}\to f^\ast Q.
    \]
    We claim that
    \[
      f^\ast Q
    \]
    is concentrated in degree $1$, and that
    \[
      f^\ast Q_{|U_\crys}=0,
    \]
    where $U_\crys\subseteq V$ denotes the crystalline part, cf.\ \Cref{section_spectrum_of_a}. We may check this on points of $V$. Let $y\in V$ be a point corresponding to a distinguished element $a\in A_E$, cf.\ \Cref{lemma_spectrum_of_a_inf}.
    By (the easier case of henselian discrete valuation rings of) the Fujiwara-Gabber theorem, cf.\ \cite[Corollary 1.18.(2)]{bhatt2018arc}, it suffices to see that the complete discretely valued field
    \[
      (A_E[1/\pi])^\wedge_a[1/a],
    \]
    has cohomological dimension $1$. This is known as its residue field is algebraically closed.\footnote{We don't know if it can happen that $f(y)$ maps to the generic point of $\tilde{V}$ for $y$ outside $U_\crys$. However, in this case $R^ij_{\ast}(\mathcal{F})_{f(y)}=0$ holds trivially for $i>0$.}
    By \Cref{section_spectrum_of_a} the remaining points lie in the crystalline part $U_\crys\subseteq V$. Now we crucially use our explicit understanding of the crystalline part $\tilde{U}_\crys$ of $\tilde{V}$, cf.\ \Cref{sec:decompleting-a_e-4-crystalline-part-of-tilde-a-e}.
    As $U_\crys$ maps to the crystalline part $\tilde{U}_\crys$ of $\tilde{V}$ it suffices to see that
    \[
      (R^ij_\ast\mathcal{F})_{x_\crys}=0
    \]
    for $i\geq 1$, and $x_\crys=\tilde{\mathfrak{p}}_\crys$ the ``crystalline point'' of $\tilde{V}$ (triviality for points mapping to the generic point of $\tilde{V}$ is clear). Again by Fujiwara-Gabber we may replace the strict henselization of $\tilde{V}$ at $x_\crys$, a valuation ring of rank $1$, by its completion. If $E$ is of characteristic $p$ we may further pass to its perfection. In the end we obtain a complete valuation ring
    \[
      R
    \]
    of rank $1$ with an algebraically closed residue field, value group $\R$. We have to see that the field
    \[
      L:=\mathrm{Frac}(R)
    \]
    is of $\ell$-cohomological dimension $0$. Now we can apply ramification theory for fields like $L$, cf.\ \cite[Chapter III]{endler1972valuation}. If $\mathrm{char}(E)=0$, then we get that $L$ is algebraically closed, and if $\mathrm{char}(E)=p>0$, then
    \[
      \mathrm{Gal}(\overline{L}/L)
    \]
    is a pro-$p$-group. As we assumed in the beginning that $\ell\neq p$ if $\mathrm{char}(E)=p$, we get the result.

    Next we claim that
    \[
      R\Gamma(\tilde{V},Q)\in D^{\leq 1}.
    \]
    By Fujiwara-Gabber and the fact that
    \[
      f^\ast Q_{|U_\crys}=0,
    \]
    we may instead prove the statement that
    \[
      H^i(V,\mathcal{F}^\prime)=0
    \]
    for each $\F_\ell$-sheaf $\mathcal{F}^\prime$ on $V$ with
    \[
      \mathcal{F}^\prime_{|U_\crys}=0.
    \]
    Using that $R\Gamma(V,-)$ commutes with filtered colimits (because $V$ is qcqs) we may replace $U_\crys$ by some open subset $W$ containing it. Then $\mathcal{F}^\prime$ is supported on finitely many points of $V\setminus U_\crys$. The statement becomes clear as each of these points has an algebraically closed residue field by \Cref{lemma_spectrum_of_a_inf}.
    This finishes the proof that
    \[
      R\Gamma(\tilde{V},Q)\leq D^{\leq 1}.
    \]
    By \Cref{sec:decompleting-a_e-1-cohomological-dimension-of-spec-tilde-a-e} below, we see that
    \[
      R\Gamma(\tilde{V},j_\ast(\mathcal{F}))\in D^{\leq 2},
    \]
    and as
    \[
      R\Gamma(\tilde{V},Q)\in D^{\leq 1}
    \] by our results above, we thus obtain
    \[
      R\Gamma(\Spec(\tilde{K}_E),\mathcal{F})=R\Gamma(\tilde{V},Rj_\ast(\mathcal{F}))\in D^{\leq 2}
    \]
    as desired.
  \end{proof}

  We cannot prove a description of $\Spec(\tilde{A}_E)$ similar to \Cref{lemma_spectrum_of_a_inf}, in particular we don't know  if $f\colon V\to \tilde{V}$ is surjective (which would imply $R^ij_\ast(\mathcal{F})=0$ for $i\geq 2$ in the above proof, making the structure of the proof a bit clearer). In a similar vein, if $a\in A_E$ is distinguished, we don't know if there exists a distinguished element $\tilde{a}\in (a)\cap \tilde{A}_E$. The problem is that we don't know if the Teichm\"uller lift
  \[
    [-]\colon \mathcal{O}_C\to A_E
  \]
  has image in $\tilde{A}_E$ (it does however when $\mathrm{char}(E)=p$).
  
\begin{proposition}
  \label{sec:decompleting-a_e-1-cohomological-dimension-of-spec-tilde-a-e}
  For each prime $\ell$ and each \'etale $\F_\ell$-sheaves $\mathcal{F}$ on $\Spec(\tilde{A}_E[1/\pi])$ 
  \[
    R\Gamma_\et(\Spec(\tilde{A}_E[1/\pi]),\mathcal{F})\in D^{\leq 2}.
  \]
  Similarly, for $\Spec(A_E[1/\pi])$.
\end{proposition}
\begin{proof}
  As $\tilde{A}_E$ is henselian along $(\pi)$ by \Cref{sec:decompleting-a_e-3-pi-completion-of-tilde-a-e} we can apply the Fujiwara-Gabber theorem, cf.\ \cite{fujiwara1995theory}, \cite[Corollary 1.18.(2)]{bhatt2018arc}, and conclude
  \[
    R\Gamma_\et(\Spec(\tilde{A}_E[1/\pi]),\mathcal{F})\cong R\Gamma_\et(\Spec(A_E[1/\pi]),\mathcal{F}) 
  \]
  for all \'etale $\F_\ell$-sheaves $\mathcal{F}$ on $\Spec(\tilde{A}_E[1/\pi])$. In the notation we suppressed here the pullback along the natural morphism
  \[
    \Spec(A_E[1/\pi])\to \Spec(\tilde{A}_E[1/\pi])
  \]
  on the RHS.
  Thus it suffices to consider the case of $A_E$.
  Set
  \[
    E_\infty
  \] as the (uncompleted) cyclotomic $\Z_p$-extension of $E$ if $\mathrm{char}(E)=0$, resp.\ the (uncompleted) perfection of $E$ if $E$ has $\mathrm{char}(E)=p$. Then we introduce
  \[
    \tilde{A}_{E_\infty}:=A_E\otimes_{\mathcal{O}_E}\mathcal{O}_{E_\infty}
  \]
  and
  \[
    A_{E_\infty}:=(\tilde{A}_{E_\infty})^\wedge_\pi.
  \]
  The ring $A_{E_\infty}$ is perfectoid, cf.\ \Cref{sec:decompleting-a_e-2-perfectoid-cover}.
  If $\mathrm{char}(E)=0$, then
  \[
    \Spec(\tilde{A}_{E_\infty}[1/\pi])\to \Spec(A_E[1/\pi]) 
  \]
  is a $\Z_p$-torsor (where $\Z_p$ is viewed as a pro-\'etale group scheme).
  If $\mathrm{char}(E)=p$, then the morphism
  \[
    \Spec(\tilde{A}_{E_\infty}[1/\pi])\to \Spec(A_E[1/\pi]) 
  \]
  is a universal homeomorphism.
  By the Fujiwara-Gabber theorem we have again
  \[
    R\Gamma(\Spec(\tilde{A}_{E_\infty}[1/\pi]),\mathcal{F})\cong R\Gamma(\Spec(A_{E_\infty}[1/\pi]),\mathcal{F}) 
  \]
  for each \'etale $\F_\ell$-sheaf $\mathcal{F}$ on $\Spec(\tilde{A}_{E_\infty}[1/\pi])$.
  Now assume that $\ell=p$. The case $\ell=p=\mathrm{char}(E)$ follows again by Artin-Schreier theory for $A_E[1/\pi]$.
  Therefore we may assume additionally that $\mathrm{char}(E)=0$. Note that then $A_{E_\infty}$ is a $p$-torsion free perfectoid ring. By \cite[Theorem 11.1.]{bhatt_scholze_prisms_and_prismatic_cohomology} we get that
  \[
    R\Gamma_\et(\Spec(A_{E_\infty}[1/\pi]),\mathcal{F})\in D^{\leq 1}
  \]
  for each \'etale $\F_p$-sheaf $\mathcal{F}$ on $\Spec(A^\prime_{E_\infty})$. Using that the $p$-cohomological dimension of $\Z_p$ is $1$, we can conclude that as desired
  \[
    R\Gamma(\Spec(A_{E}[1/\pi]),\mathcal{F})\in D^{\leq 2}
  \]
  for each \'etale $\F_p$-sheaf $\mathcal{F}$ on $\Spec(A_{E}[1/\pi])$.
  This finishes the case $\ell=p$, and we pass to the case $\ell\neq p$.
  As before, we want to prove that
  \[
    R\Gamma(\Spec(A_{E}[1/\pi]),\mathcal{F})\in D^{\leq 2}
  \]
  for each \'etale $\F_\ell$-sheaf $\mathcal{F}$ on $\Spec(A_E[1/\pi])$.
  Let us denote by $A^\prime_E$ the ring $A_E$ \textit{equipped with the $\pi$-adic topology}. Similarly, we introduce $A_{E_\infty}^\prime$. By \Cref{sec:gener-affin-comp-1-affinoid-comparison-theorem}
  \[
    R\Gamma_\et(\Spec(A_{E}[1/\pi]),\mathcal{F})\cong R\Gamma_{\et}(\Spd(A^\prime_{E_\infty}[1/\pi],A^\prime_{E}),\mathcal{F}^{\mathrm{ad}})
  \]
  and $\mathcal{F}$ any \'etale $\F_\ell$-sheaf on $\Spec(A_{E_\infty}[1/\pi])$. We claim that 
  \[
    H^i(\Spd(A_E^\prime[1/\pi],A_E^\prime),\mathcal{F})=0
  \]
  for $i>2$ and any \'etale $\F_\ell$-sheaf $\mathcal{F}$ on the spatial diamond \[
    V_E:=\Spd(A_E^\prime[1/\pi],A_E^\prime).
  \]
  By \cite[Proposition 21.11.]{scholze_etale_cohomology_of_diamonds} it suffices to show that
  \[
    \mathrm{dim}(V_E)=1
  \]
  (where the dimension is defined as in \cite[Chapter 21]{scholze_etale_cohomology_of_diamonds}, i.e., the Krull dimension of the underlying spectral space of $V_E$), and
  \[
    \mathrm{cd}_\ell(y)\leq 1
  \]
  for each maximal point $y\in V_E$, cf.\ \cite[Definition 21.10.]{scholze_etale_cohomology_of_diamonds}.
  We may calculate $\mathrm{dim}(V_E)$ via the $\Z_p$-cover
  \[
    g \colon V_{E_\infty}:=\Spd(A_{E_\infty}^\prime[1/\pi],A_{E_\infty}^\prime) \to V_{E},
  \]
  where $V_{E_\infty}$ is a perfectoid space.
  Let $\varpi\in \mathcal{O}_C$ be a pseudo-uniformizer. Then by tilting (which amounts to replacing $E_\infty$ by its tilt) we see that the open subset
  \[
    V_{E_\infty}^{\circ}:=\bigcup\limits_{n\in \N}\{|[\varpi]|\geq |\pi|^n\}\subseteq V_{E_\infty}
  \]
  is the (diamond associated to the) perfection of the punctured open unit disc over $\Spa(C,\mathcal{O}_C)$, and thus of dimension $1$.
  In the following we will use freely that the topological spaces of
  \[
    V_E
  \]
  and
  \[
    \Spa(A^\prime_E[1/\pi],A_E^\prime)
  \]
  coincide, cf.\ \cite[Lemma 15.6.]{scholze_etale_cohomology_of_diamonds}.
  The complement of
  \[
    V_{E}^\circ:=g(V_{E_\infty}^\circ)
  \]
  in $V_{E}$ has three points, first the ``crystalline'' point $y_{\crys}$ which is the image of
  \[
    \Spd(E,\mathcal{O}_E)\subseteq V_{E}
  \]
  (and this is the only point $y\in V_{E_\infty}$ such that $\varpi$ maps to $0$ in the residue field of $y$). On the other hand, there is the image $y_{1,\mathrm{add}}$ of the $\pi$-adic valuation in
  \[
    \Spd((A_E^\prime[1/[\varpi]])^\wedge_\pi[1/\pi], (A_E^\prime[1/[\varpi]])^\wedge_\pi)
  \]
  (the field $(A_E[1/[\varpi]])^\wedge_\pi[1/\pi]$ is $E\otimes_{W(k)}W(C)$ if $\mathrm{char}(E)=0$, and $C((\pi))$ otherwise) under the canonical morphism to $\Spd(A_E^\prime[1/\pi],A_E^\prime)$.
  The point $y_{1,\mathrm{add}}$ is the unique maximal point $y\in V_E$ such that $\varpi$ maps to a unit in the residue field at $y$ (seen as a point in $\Spa(A^\prime_E[1/\pi],A_E^\prime)$). Let us note that $y_{1,\mathrm{add}}$ has one specialization $y_{2,\mathrm{add}}$ in $V_{E}$. Namely, as $\Spa(A^\prime_E[1/\pi],A_E^\prime)$ is analytic specializations in it do not change the residue field. Thus, specializations of $y_{1,\mathrm{add}}$ correspond to valuation rings in $C$ containing $\mathcal{O}_C$. As $\mathcal{O}_C$ is of rank $1$, these valuation rings are $C$ (corresponding to $y_{1,\mathrm{add}}$) and $\mathcal{O}_C$ (corresponding to $y_{2,\mathrm{add}}$). 
  Note that the points $y_{1,\mathrm{add}}, y_{2,\mathrm{add}}$ make the difference between
  \[
    \Spa(A_E,A_E)\setminus V([\varpi])
  \]
  (here $A_E$ carries the $(\pi,[\varpi])$-adic topology)
  and
  \[
    \Spa(A^\prime_E[1/\pi])
  \]
  (here $A_E^\prime$ carries the $\pi$-adic topology).
  Altogether, we see that
  \[
    \mathrm{dim}(V_E)=1.
  \]
  Next we show that each maximal point $y\in V_E$ has $\ell$-cohomological dimension (in the sense of \cite[Definition 21.10]{scholze_etale_cohomology_of_diamonds})
  \[
    \mathrm{cd}_{\ell}(y)\leq 1.
  \]
  Indeed, as $\ell\neq p$ this holds for all points in $V_{E}^\circ$ because the same is true for its $\Z_p$-cover $V_{E_\infty}^\circ$ (by tilting and \cite[Proposition 21.16.]{scholze_etale_cohomology_of_diamonds}). For the remaining maximal points
  \[
    y_{\crys}, y_{1,\mathrm{add}}
  \]
  this holds true as finite extensions of the fields
  \[
    W(L)[1/p]
  \]
  resp.\
  \[
    L((\pi))
  \]
  for any algebraically closed field $L$ of characteristic $p$ are of cohomological dimension $1$.
  As announced we can apply \cite[Proposition 21.11.]{scholze_etale_cohomology_of_diamonds} to
  \[
    V_E
  \]
  and conclude. This finishes the proof.
\end{proof}

Note \cite[Proposition 21.11.]{scholze_etale_cohomology_of_diamonds} holds true for $\ell=p$ as well. However, we needed a different argument in that case as we cannot conclude
\[
  \mathrm{cd}_p(y)\leq 1
\]
(only $\mathrm{cd}_p(y)\leq 2$) for points in $V_E^\circ$ by tilting the $\Z_p$-cover $V_{E_\infty}^\circ$ of $V_E^\circ$. 

In the proof we used the following statement, which appears in \cite[Proposition 13.1.1.]{scholze2020berkeley}, too.

\begin{lemma}
  \label{sec:decompleting-a_e-2-perfectoid-cover}
  The ring
  \[
      A^\prime_{E_\infty}=(\tilde{A}^\prime_{E_\infty})^\wedge_\pi
    \]
    is perfectoid.
  \end{lemma}
  \begin{proof}
    As $A^\prime_{E_\infty}$ is an algebra over the perfectoid ring ${\mathcal{O}_{E_\infty}}^{\wedge}_\pi$ it suffices to see that $A^\prime_{E_\infty}$ is $\pi$-torsion free and that the Frobenius on
    \[
      A^\prime_{E_\infty}/\pi
    \]
    is surjective. But
    \[
           \tilde{A}^\prime_{E_\infty}/\pi\cong \mathcal{O}_C\otimes_{k} \mathcal{O}_{E_\infty}/\pi, 
         \]
         and thus the Frobenius is indeed surjective.
  \end{proof}

  If $\mathrm{char}(E)=p$, then we need the following additional result on the dimension of $\tilde{K}_E$, cf.\ \cite[\S 10.1.]{serre1995cohomologie}.
  Let $F$ be a field of characteristic $p$, set
  \[
    \Omega^i:=\Lambda^i\Omega^1_{F/\Z}
  \]
  and denote by $d\colon \Omega^{i-1}\to \Omega^i$ the exterior derivative. Let
  \[
    \gamma\colon \Omega^i\to \Omega^i/{d\Omega^{i-1}}
  \]
  be the inverse Cartier operator.
  Then set
  \[
    H^{i+1}_p(F):=\mathrm{coker}(\gamma-1\colon \Omega^{i}\to \Omega^{i}/d\Omega^{i-1})
  \]
  for $i\geq 0$ (with $\Omega^{-1}:=\{0\}$).
  To apply Gille's result on Serre's conjecture, cf.\ \cite{gille_cohomologie_galoisienne_des_groupes_quasi_deployes_sur_des_corps_de_dimension_cohomologique_2}, we need the following vanishing result for $H^\ast_p$.

  \begin{proposition}
    \label{sec:decompleting-a_e-1-kato-cohomologie-of-tilde-k-e}
    Assume $\mathrm{char}(E)=p>0$. For each finite separable extension $F$ of $\tilde{K}_E$ we have
    \[
      H^i_p(F)=0
    \]
    for $i\geq 3$.
  \end{proposition}
  \begin{proof}
    Let $E_\infty$ be the perfection of $E$.
    Considering Mal'cev-Neumann series we see that the (uncompleted) tensor product
    \[
      \mathcal{O}_{E_\infty}\otimes_{\mathcal{O}_E} \tilde{A}_E
    \]
    is a perfect ring. Thus,
    \[
      E_\infty\otimes_E \tilde{K}_E
    \]
    is a perfect field, and in fact the perfection of $\tilde{K}_E$. Similarly, as $F$ is a separable extension of $\tilde{K}_{E}$,
    \[
      F_\infty:=E_\infty\otimes_E F
    \]
    is the perfection of $F$ (note that $E$ is integrally closed in $\tilde{K}_E$ as follows from the existence of a morphism $\tilde{A}_E\to E$ of $E$-algebras).
    In particular, we see that each finite separable extension $F$ of $\tilde{K}_E$ satisfies
    \[
      [F:F^p]= p.
    \]
    Now we can apply \Cref{sec:decompleting-a_e-1-kato-cohomology-for-fields-with-p-bases}.
  \end{proof}

  \begin{lemma}
    \label{sec:decompleting-a_e-1-kato-cohomology-for-fields-with-p-bases}
    Let $F$ be a field of characteristic $p$, such that $[F:F^p]=p^r$ for some $r\geq 0$.
    Then $H^i_p(F)=0$ for $i\geq r+2$.
  \end{lemma}
  \begin{proof}
    The field $F$ admits a $p$-basis of $r$ elements. In particular,
    \[
      \Omega^1_{F/\F_p}
    \]
    is free of rank $r$ over $F$, and thus
    \[
      \Omega^j_{F/\F_p}=0
    \]
    for $j\geq r+1$.
    Because
    \[
      H^{i}_p(F)=\mathrm{coker}(\gamma-1\colon \Omega^{i-1}_{F/\F_p}\to \Omega^{i-1}_{F/\F_p}/d\Omega^{i-2}_{F/\F_p}),
    \]
    we obtain
    \[
      H^i_p(F)=0
    \]
    for $i\geq r+2$.
  \end{proof}
  
\section{Generalities on torsors}
\label{section_generalities_on_torsors}

In this section we collect some general facts about torsors we will use later.
The following theorem of Steinberg will be important for us.

\begin{theorem}
\label{theorem_steinbergs_theorem}  
Let $K$ be of dimension $\leq 1$, i.e., for every finite field extension $K^\prime/K$ the Brauer group $\mathrm{Br}(K^\prime)$ vanishes.
Then for every (connected) reductive group $G/K$ the cohomology set $H^1(K,G)=\{1\}$ is trivial.
\end{theorem}
\begin{proof}
 This is \cite[Chapitre III.2.3, Th\'eor\`eme $1^\prime$]{serre_cohomologie_galoisienne} (noting \cite[Remarques 1)]{serre_cohomologie_galoisienne} following it).
\end{proof}

For example, fields complete under a discrete valuation whose residue field is algebraically closed are of dimension $1$ (cf.\ \cite[Chapitre II.3.3.c)]{serre_cohomologie_galoisienne}).

We now want to discuss shortly the Beauville-Laszlo glueing for torsors. Thus we consider the following situation.

Let $A$ be a ring and let $f\in A$ be a non-zero divisor.
Let $A_f$ be the localisation of $A$ at $f$ and let $\widehat{A}$ be the $f$-adic completion of $A$.
Moreover, let 
$$
\mathcal{G}\to \Spec(A)
$$ 
be an affine, flat group scheme over $A$. Then we have the following immediate consequence of the Beauville-Laszlo glueing lemma (cf.\ \cite{beauville_laszlo_un_lemme_de_descente}).

In the following, ``torsor'' means ``torsor for the fpqc-topology''.

\begin{lemma}
\label{lemma_beauville_laszlo_for_torsors}
Sending a $\mathcal{G}$-torsor $\mathcal{P}$ on $\Spec(A)$ to 
$$
(\mathcal{P}_1:=\mathcal{P}_{|\Spec(A_f)},\mathcal{P}_2:=\mathcal{P}_{|\Spec(\widehat{A})},\alpha\colon {\mathcal{P}_1}_{|\Spec(\widehat{A}[1/f])}\cong\mathcal{P}_{|\Spec(\widehat{A}[1/f])}\cong {\mathcal{P}_2}_{|\Spec(\widehat{A}[1/f]})
$$
defines an equivalence between the groupoid of $\mathcal{G}$-torsors on $\Spec(A)$ and the category of triples 
$$
(\mathcal{P}_1,\mathcal{P}_2,\alpha)
$$ 
with $\mathcal{P}_1$ a $\mathcal{G}_{|\Spec(A_f)}$-torsor on $\Spec(A_f)$, $\mathcal{P}_2$ a $\mathcal{G}_{|\Spec(\widehat{A})}$-torsor on $\Spec(\widehat{A})$ and $\alpha\colon {\mathcal{P}_1}_{|\Spec(\widehat{A}[1/f])}\cong {\mathcal{P}_2}_{|\Spec(\widehat{A}[1/f])}$ an isomorphism.
\end{lemma}
\begin{proof}
From \cite{beauville_laszlo_un_lemme_de_descente} one can conclude that the category of flat $A$-modules $M$ is equivalent to the category of triples
$$
(M_f,\widehat{M},\alpha)
$$
with $M_f$ a flat $A_f$-module, $\widehat{M}$ a flat $\widehat{A}$-module and $\alpha\colon M_f\otimes_{A_f} \widehat{A}[1/f]\cong \widehat{M}[1/f]$ an isomorphism. This equivalence respects tensor products and hence induces an equivalence on algebra/coalgebra objects. Moreover, a faithfully flat affine scheme $X$ over $\Spec(A)$ with an action by $\mathcal{G}$ is a $\mathcal{G}$-torsor for the fpqc-topology if and only if the canonical morphism
$$
\mathcal{G}\times_{\Spec(A)}X\to X\times_{\Spec(A)}X,\ (g,x)\mapsto (gx,x)
$$
is an isomorphism. This condition can be phrased in terms of coordinate rings and hence we obtain the lemma.
\end{proof}

If $\mathcal{G}$ is smooth over $\Spec(A)$, then every fpqc-torsor is actually trivial for the \'etale topology and we obtain \Cref{lemma_beauville_laszlo_for_torsors} with ``fpqc'' replaced by ``\'etale''. In fact, if $\mathcal{G}$ is smooth, then every $\mathcal{G}$-torsor $\mathcal{P}$ is smooth and thus admits sections \'etale locally.

For $A=A_E$ (equipped with the $(\pi,[\varpi])$-adic topology for some pseudo-uniformizer $C$) we shortly discuss a comparison for ``algebraic'' and ``adic torsors''. We include this statement only for completeness as it is not needed in the sequel. We recommend \cite[Appendix to lecture XIX]{scholze2020berkeley} for a discussion of torsors over adic spaces.

\begin{proposition}
\label{proposition_comparison_algebraic_adic_torsors}
Let $s\in \Spec(A_E)$ resp.\ $s^\prime\in \mathrm{Spa}(A_E,A_E)$ be the closed point, where $\Spa(A_E,A_E)$ denotes the adic spectrum of $A_E$. Then for every affine smooth group scheme $\mathcal{G}/\mathcal{O}_E$ the groupoids of 
$\mathcal{G}$-torsors on $U:=\Spec(A_E)\setminus\{s\}$ and $\mathcal{G}^{\mathrm{adic}}$-torsors on $\mathcal{U}:=\mathrm{Spa}(A_E,A_E)\setminus\{s^\prime\}$ are canonically equivalent.
\end{proposition}
\begin{proof}
  By \cite[Theorem 14.2.1.]{scholze2020berkeley} and
  \cite[Lemma 14.2.3.]{scholze2020berkeley}, cf.\ \Cref{lemma-extending-vector-bundles} there are natural equivalences
$$
\mathrm{Bun}(\mathcal{U})\cong\mathrm{Bun}(\Spec(A_E))
$$
and
$$
\mathrm{Bun}(U)\cong \mathrm{Bun}(\Spec(A_E))
$$
(the same proof works if $E$ is of equal characteristic).
Using \cite[Theorem 19.5.2.]{scholze2020berkeley} we therefore have to prove that 
$$
\mathrm{Bun}(\mathcal{U})\cong \mathrm{Bun}(U)
$$
as \textit{exact} categories, namely by \cite[Theorem 19.5.2.]{scholze2020berkeley} the groupoid of $\mathcal{G}$-torsors identifies with the groupoid of fiber functors on $\mathrm{Rep}_{\mathcal{O}_E}(\mathcal{G})$ over $U$ resp.\ $\mathcal{U}$. Let $u:=[\varpi]\in A_E$ be the Teichm\"uller lift of some $\varpi\in \mathfrak{m}_{C}\setminus\{0\}$.
If
$$
0\to \mathcal{M}_1\to \mathcal{M}_2\to \mathcal{M}_3\to 0
$$
is an exact sequence of vector bundles on $U$, then setting $M_i=H^0(U,\mathcal{M}_i)$ we obtain an exact sequence
$$
0\to M_1\to M_2\to M_3\to Q\to 0
$$
of $A_E$-modules where the finitely presented $Q$ is killed by some power of the ideal $(\pi,u)$.
For every affinoid $\Spa(B,B^+)\subseteq \mathcal{U}$ either $\pi$ or $u$ is invertible on $B$. In particular,
$$
\mathrm{Tor}_i^{A_E}(B,Q)=0
$$ 
for $i\geq 0$, which shows that
$$
0\to B\otimes_{A_E}M_1\to B\otimes_{A_E} M_2\to B\otimes_{A_E} M_3\to 0
$$
is exact.
This proves that the functor
$$
\mathrm{Bun}(U)\to \mathrm{Bun}(\mathcal{U})
$$
is exact.
Conversely, assume that
$$
0\to \mathcal{N}_1\to \mathcal{N}_2\to \mathcal{N}_3\to 0
$$
is an exact sequence of vector bundles on $\mathcal{U}$. Let $N_i=H^0(\mathcal{U},\mathcal{N}_i)$ be the associated finite free $A_E$-modules under the equivalence $\mathrm{Bun}(\mathcal{U})\cong \mathrm{Bun}(\Spec(A_E))$ from \cite[Theorem 14.2.1.]{scholze2020berkeley}. 
Set
$$
\begin{matrix}
\mathcal{U}_1:=\{\ |\pi|\leq |u|\neq 0\}\subseteq \mathcal{U} \\
\mathcal{U}_2:=\{\ |u|\leq |\pi|\neq 0\}\subseteq \mathcal{U} \\
\mathcal{U}_{12}:=\mathcal{U}_1\cap \mathcal{U}_2.
\end{matrix}
$$
By definition we obtain a diagram with exact rows and columns
$$
\xymatrix{
 &0\ar[d] &0\ar[d] &0\ar[d] & \\
0\ar[r] &N_1\ar[r]\ar[d] & N_2\ar[r]\ar[d] & N_3\ar[d] &  \\
0\ar[r] &{H^0(\mathcal{U}_1,\mathcal{N}_1) \atop \oplus H^0(\mathcal{U}_2,\mathcal{N}_1)}\ar[r]\ar[d] & {H^0(\mathcal{U}_1,\mathcal{N}_2)\atop \oplus H^0(\mathcal{U}_2,\mathcal{N}_2)}\ar[r]\ar[d] & {H^0(\mathcal{U}_1,\mathcal{N}_3)\atop\oplus H^0(\mathcal{U}_2,\mathcal{N}_3)}\ar[r]\ar[d] & 0 \\
0\ar[r] &H^0(\mathcal{U}_{12},\mathcal{N}_1)\ar[r]\ar[d] & H^0(\mathcal{U}_{12},\mathcal{N}_2)\ar[r]\ar[d] & H^0(\mathcal{U}_{12},\mathcal{N}_3)\ar[r]\ar[d] & 0\\
 &H^1(\mathcal{U},\mathcal{N}_1)\ar[r]\ar[d] &H^1(\mathcal{U},\mathcal{N}_2)\ar[r]\ar[d] & H^1(\mathcal{U},\mathcal{N}_3)\ar[r]\ar[d] & 0\\
& 0 & 0 & 0.
}
$$
After inverting $\pi$ the groups $H^1(\mathcal{U},\mathcal{N}_i)$ vanish. In fact, they are (as in the proof of \cite[Theorem 14.2.1.]{scholze2020berkeley}) given by $H^1(\Spa(\tilde{A}[1/\pi]),\tilde{\mathcal{N}}_i)$ with $\tilde{\mathcal{N}}_i$ a vector bundle on the affinoid adic space 
$$
\Spa(\tilde{A}[1/\pi])=\{\ x\in \Spa(\tilde{A},\tilde{A})\ |\ |\pi(x)|\neq 0 \},
$$ where $\tilde{A}=A_E$ but equipped with the $\pi$-adic topology. Namely, consider the spaces 
$$
\tilde{\mathcal{U}}_1:=\{|\pi|\leq |u|\neq 0\}\subseteq \Spa(\tilde{A}[1/\pi])
$$
and
$$
\tilde{\mathcal{U}}_2:=\{|u|\leq |\pi|\neq 0\}\subseteq \Spa(\tilde{A}[1/\pi]).
$$
Then
$$
\tilde{\mathcal{U}}_2\cong \mathcal{U}_2
$$
and 
$$
H^0(\tilde{\mathcal{U}}_1,\mathcal{O}_{\tilde{\mathcal{U}}_1})\cong H^0(\mathcal{U}_1,\mathcal{O}_{\mathcal{U}_1})[1/\pi]
$$
by \cite[Lemma 14.3.1.]{scholze2020berkeley}.
Moreover,
$$
\tilde{\mathcal{U}}_1\cap\tilde{\mathcal{U}}_2\cong \mathcal{U}_1\cap\mathcal{U}_2.
$$
Thus the modules $H^0(\mathcal{U}_1,\mathcal{N}_i)[1/\pi]$ and $H^0(\mathcal{U}_2,\mathcal{N}_i)$ glue to a vector bundle $\tilde{\mathcal{N}}_i$ on $\Spa(\tilde{A}[1/\pi])$
as
$$
\Spa(\tilde{A}[1/\pi])
$$
is sheafy (cf.\ \cite[Proof of Proposition 13.1.1.]{scholze2020berkeley}). By sheafiness we can thus conclude that 
$$
H^1(\mathcal{U},\mathcal{N}_i)[1/\pi] \cong H^1(\Spa(\tilde{A}[1/\pi]),\tilde{\mathcal{N}}_i) =0
$$ 
(cf.\ \cite[Theorem 5.2.6.]{scholze2020berkeley}).
This implies that the cokernel $Q$ of $N_2\to N_3$ is $\pi$-torsion. Hence, to show that it vanishes on $U$ it suffices to prove that 
$$
Q\otimes_{A_E}R=0
$$
where $R$ is the $\pi$-adic completion of $A_E[1/u]$. But $R$ is flat over $A_E$ and by assumption the sequence
$$
0\to N_1\otimes_{A_E} R\to N_2\otimes_{A_E} R\to N_3\otimes_{A_E} R\to 0
$$
is exact it identifies with the $\pi$-adic completion of the stalk at $(\pi)\in \mathcal{U}$ of the exact sequence
$$
0\to \mathcal{N}_1\to \mathcal{N}_2\to \mathcal{N}_3\to 0.
$$ 
This finishes the proof.
\end{proof}

\section{Generalities on extending torsors}
\label{section_generalities_on_extending_torsors}

We want to draw some consequences of \Cref{lemma-extending-vector-bundles} in a more abstract setup.
For this let $A$ be any ring and let $U\subseteq \Spec(A)$ be a quasi-compact open subset.
We assume that the restriction functor for vector bundles
$$
\mathrm{Bun}(\Spec(A))\xrightarrow{\cong} \mathrm{Bun}(U)
$$
is an equivalence.

For example, $A$ can be $A_E$ (cf.\ \Cref{lemma-extending-vector-bundles}) or a two-dimensional regular local ring (cf.\ \Cref{lemma_vector_bundles_extend_noetherian_case}).
Even under this general assumption, we can conclude that the functor inverse to restriction must send a vector bundle $\mathcal{V}$ on $U$ to the quasi-coherent module on $\Spec(A)$ associated with the $A$-module
$$
H^0(U,\mathcal{V})
$$
because if $\mathcal{V}^\prime$ denotes a preimage of $\mathcal{V}$, i.e., $\mathcal{V}^\prime$ is a vector bundle on $\Spec(A)$ and $\mathcal{V}^\prime_{|U}\cong \mathcal{V}$, then
$$
H^0(U,\mathcal{V})= \mathrm{Hom}_{\mathcal{O}_U}(\mathcal{O}_U,\mathcal{V})\cong
\mathrm{Hom}_{\mathcal{O}_{\Spec(A)}}(\mathcal{O}_{\Spec(A)},\mathcal{V}^\prime(\Spec(A)))\cong \mathcal{V}^\prime(\Spec(A)).
$$
In particular, we obtain that for $\mathcal{V}$ a vector bundle on $U$ the global sections $H^0(U,\mathcal{V})$ are finite locally free over $A$.

\begin{corollary}
\label{corollary:morphisms_to_affine_schemes_extend}
Let $A$ be as above and let $f\colon X\to \Spec(A)$ be an affine morphism and let $g\colon U\to X_U:=X\times_{\Spec(A)}U$ be a section of $X$ over $U$. Then $g$ extends uniquely to a section $g^\prime\colon \Spec(A)\to X$.   
\end{corollary}
\begin{proof}
This is a consequence of the adjunction for the $\Spec$-functor. Namely write $X=\Spec(B)$ for some $A$-algebra $B$. Then
$$
\begin{array}{ll}
 &\mathrm{Hom}_{\Spec(A)}(U,X) \\
\cong &  \mathrm{Hom}_A(B,\Gamma(U,\mathcal{O}_U))\\
\cong & \mathrm{Hom}_A(B,A)\\
\cong & \mathrm{Hom}_{\Spec(A)}(\Spec(A),\Spec(B))
\end{array}
$$
because $\Gamma(U,\mathcal{O}_U)\cong A$.
\end{proof}

\begin{proposition}
\label{proposition:restricting-flat-schemes-is-fully-faithful}
Let $A$ be as above.
Then the base change functor
$$
\begin{matrix}
\{X\to \Spec(A) \textrm{ affine and flat }\} &\to & \{Y\to U \textrm{ affine and flat }\}\\
X & \mapsto & X\times_{\Spec(A)}U 
\end{matrix}
$$
is fully faithful and its essential image consists of all affine and flat morphisms $g\colon Y\to U$ such that $H^0(Y,\mathcal{O}_Y)$ is a flat $A=H^0(U,\mathcal{O}_U)$-module.
\end{proposition}
\begin{proof}
We prove that the base change $X\mapsto X\times_{\Spec(A)}U$ induces an equivalence of the categories
$$
\mathcal{C}:=\{X\to \Spec(A) \textrm{ affine and flat}\}
$$ 
and
$$
\mathcal{D}:=\{g\colon Y\to U \textrm{ affine and flat such that } H^0(Y,\mathcal{O}_Y) \textrm{ is flat over } A \}.
$$ 
In fact, for $g\colon Y\to U$ in $\mathcal{D}$ we set
$$
X:=\Spec(H^0(Y,\mathcal{O}_Y))
$$
for the flat $A$-algebra $H^0(Y,\mathcal{O}_Y)$, which defines a functor from $\mathcal{D}$ to $\mathcal{C}$.
For 
$$
f\colon X\to \Spec(A)
$$ 
affine and flat with base change $g\colon X\times_{\Spec(A)}U\to U$ let us write 
$$
f_\ast\mathcal{O}_X=\varinjlim \mathcal{V}_i
$$ 
for vector bundles $\mathcal{V}_i$ on $\Spec(A)$ (using Lazard's theorem \cite[Tag 058G]{stacks_project}) and compute
$$
H^0(U,g_\ast(\mathcal{O}_X))= H^0(U,\varinjlim \mathcal{V}_{i|U})=\varinjlim \mathcal{V}_i(\Spec(A))=\mathcal{O}_X(X)
$$
where we used our assumption on $A$ and that $U$ and $g$ are quasi-compact (and quasi-separated) to commute global sections resp.\ the direct image and filtered colimits.
On the other hand, if we start with 
$$
g\colon Y\to U
$$ 
in $\mathcal{D}$, then we obtain a canonical morphism
$$
Y\to \Spec(H^0(U,g_\ast(\mathcal{O}_Y)))
$$
over $\Spec(A)$ which restricts to the isomorphism (as $g\colon Y\to U$ is affine)
$$
Y\cong \underline{\Spec}(g_\ast(\mathcal{O}_Y))
$$
over $U$. This finishes the proof.
\end{proof}

Note that each affine $U$-scheme $Y$ extends to an affine $\Spec(A)$-scheme by considering its global sections. However, two non-isomorphic, integral affine schemes $X,X^\prime$ over $\Spec(A)$ can have isomorphic restriction to $U$. As a concrete example, let 
$$
A^\prime=k[x,y]
$$ 
be the polynomial ring and consider the $A^\prime$-algebra $$
B^\prime:=k[u,v]
$$ 
with $u=x$ and $v=\frac{y}{x}$ (an affine chart of the blow-up of $A^\prime$ in $(x,y)$). Set $A$ as the localization of $A^\prime$ at $(x,y)$ and $B:=A\otimes_{A^\prime} B^\prime$. Then $\Spec(B)\to \Spec(A)$ is isomorphic to $\Spec(A[1/x])\to \Spec(A)$ when restricted to $U=\Spec(A)\setminus\{(x,y)\}$. Here, $B$ is not flat over $A$ as the dimension of the fibers jumps.
Examples of affine, flat schemes over $U$, which do not extend to affine, \textit{faithfully flat} schemes over $\Spec(A)$ are provided by \Cref{example:non-trivial-torsors-non-parahoric-group-scheme}).

Using \Cref{proposition:restricting-flat-schemes-is-fully-faithful} we arrive at the following descent criterion for extending some affine and flat morphism $Y\to U$ to $\Spec(A)$.

Now let $A\to A^\prime$ be some morphism such that the restriction functor
$$
\mathrm{Bun}(\Spec(A^\prime))\xrightarrow{\cong} \mathrm{Bun}(U^\prime)
$$
of vector bundles from $\Spec(A)$ to $U^\prime:=U\times_{\Spec(A)}\Spec(A^\prime)$ is an equivalence. 

\begin{lemma}
\label{lemma-descent-criterion-for-extending}
With the notations from above assume furthermore that $A^\prime$ is faithfully flat over $A$. Let $Y\to U$ be an affine and faithfully flat morphism and assume that 
$$
Y\times_U U^\prime\to U^\prime
$$ 
extends to some affine faithfully flat $\Spec(A^\prime)$-scheme. Then $Y$ extends to an affine faithfully flat $\Spec(A)$-scheme.  
\end{lemma}
\begin{proof}
Indeed, by \Cref{proposition:restricting-flat-schemes-is-fully-faithful} we have to check whether the global sections
$$
H^0(Y,\mathcal{O}_Y)
$$
are faithfully flat over $A$. But this can be checked after the faithfully flat base change $A\to A^\prime$ and over $A^\prime$ it holds true after assumption and \Cref{proposition:restricting-flat-schemes-is-fully-faithful}.
\end{proof}

Now let $\mathcal{G}$ be an affine flat group scheme over $\Spec(A)$.

\begin{lemma}
\label{lemma_fully_faithfulness_of_torsors_general_case}
The restriction functor from $\mathcal{G}$-torsors on $\Spec(A)$ to $\mathcal{G}_{|U}$-torsors on $U$ is fully faithful with essential image given by $\mathcal{G}_{|U}$-torsors $\mathcal{P}$ on $U$ such that the global sections $H^0(\mathcal{P},\mathcal{O}_{\mathcal{P}})$ are faithfully flat over $A$.  
\end{lemma}
\begin{proof}
  Each $\mathcal{G}$-torsor on $\Spec(A)$ is represented by some affine, faithfully flat scheme over $\Spec(A)$. Therefore we may apply \Cref{proposition:restricting-flat-schemes-is-fully-faithful} to conclude fully faithfulness of restrictions of $\mathcal{G}$-torsors. Let $\mathcal{P}$ be a $\mathcal{G}_{|U}$-torsor on $U$ such that the global sections of $\mathcal{P}$ are faithfully flat over $A$. By \Cref{proposition:restricting-flat-schemes-is-fully-faithful} the underlying scheme of $\mathcal{P}$ extends to a faithfully flat scheme $\mathcal{P}^\prime$ over $\Spec(A)$. By fully faithfulness of restrictions the $\mathcal{G}_{|U}$-action extends to $\mathcal{P}^\prime$. That $\mathcal{P}^\prime$ is a $\mathcal{G}$-torsor can then again be checked after restricting to $U$. This finishes the proof. 
\end{proof}

In the reductive case $\mathcal{G}$-torsors on $U$ will automatically extend to $\Spec(A)$. We recall the argument of Colliot-Th\'el\`ene and Sansuc (cf.\ \cite[Th\'eor\`eme 6.13]{colliot_thelene_sansuc_fibres_quadratiques_et_composantes_connexes_reelles}) (at least if $A$ is local).

\begin{proposition}
  \label{proposition_torsors_general_split_case}
  Assume that $A$ is local and that $\mathcal{G}$ is a reductive group scheme over $A$.
  Let $\mathcal{P}$ be a $\mathcal{G}_{|U}$-torsor on $U$. Then $\mathcal{P}$-extends (uniquely) to $\Spec(A)$.
\end{proposition}
\begin{proof}
  By \cite[Corollary 9.7.7.]{alper_adequate_moduli_spaces_and_geometrically_reductive_group_schemes} there exists an embedding $\mathcal{G}\hookrightarrow \mathrm{GL}_n$ such that the quotient $\mathrm{GL}_n/\mathcal{G}$ is affine. Now consider the exact sequence (of cohomology sets taken for the \'etale topology)
  $$
H^0(U,\mathrm{GL}_n/\mathcal{G})\overset{\delta}{\to} H^1(U,\mathcal{G})\to H^1(U,\mathrm{GL}_n).
$$
As we assumed that $A$ is local $H^1(U,\mathrm{GL}_n)=\{1\}$ and thus the class of $\mathcal{P}\in H^1(U,\mathcal{G})$ lies in the image of $\delta$. But as $\mathrm{GL}_n/\mathcal{G}$ is affine
$$
H^0(U,\mathrm{GL}_n/\mathcal{G})=H^0(\Spec(A),\mathrm{GL}_n/\mathcal{G})
$$
by \Cref{corollary:morphisms_to_affine_schemes_extend} and thus the morphism $\delta$ factors as desired over $H^1(\Spec(A),\mathcal{G})$ by naturality of the connecting morphism.
\end{proof}

\section{Extending torsors to $\mathrm{Spec}(A_E)$}
\label{sec:extend-tors-mathrmsp}

In this section we want to prove \Cref{theorem_introduction_main_theorem_a}, i.e., that torsors under parahoric group schemes over $\mathcal{O}_E$ on the punctured spectrum
$$
U_E=\mathrm{Spec}(A_E)\setminus\{s\}
$$
of $A_E$ extend to $\Spec(A_E)$.
We continue to use the notation from \Cref{sec:notations}. In particular, we use the notation $\mathcal{G}/\mathcal{O}_E$ for a parahoric group scheme over $\mathcal{O}_E$ with reductive generic fiber $G/E$.

From our general discussion of extending torsors and \Cref{lemma-extending-vector-bundles} we can conclude that the restriction functor
$$
\{\mathcal{G}\textrm{-torsors on } \Spec(A_E)\}\to \{\mathcal{G}\textrm{-torsors on } U\}
$$
is fully faithful and an equivalence if $\mathcal{G}$ is reductive (cf.\ \Cref{lemma_fully_faithfulness_of_torsors_general_case} and \Cref{proposition_torsors_general_split_case}).

For the moment let us shortly denote $A_E$ by $A_{E,C}$ and fix an extension $C^\prime/C$ of perfect non-archimedean fields over $k$. Set
$$
A_{E,C^\prime}:=W(\mathcal{O}_{C^\prime})\hat{\otimes}_{W(k)}\mathcal{O}_E.
$$
We obtain the following descent statement.

\begin{lemma}
  \label{lemma_descent_along_extension_of_c}
 A $\mathcal{G}$-torsor $\mathcal{P}$ on the punctured spectrum of $\Spec(A_{E,C})$ extends to $\Spec(A_E)$ if the base change of $\mathcal{P}$ to the punctured spectrum of $\Spec(A_{E,C^\prime})$ does.
\end{lemma}
\begin{proof}
  By \cite[Theorem 19.5.1.]{scholze2020berkeley} the torsor $\mathcal{P}$ defines an exact tensor functor
  $$
\omega\colon \mathrm{Rep}_{\mathcal{O}_E}(\mathcal{G})\to \mathrm{Bun}(U)
$$
from the category $\mathrm{Rep}_{\mathcal{O}_E}(\mathcal{G})$ of representations of $\mathcal{G}$ on finite free $\mathcal{O}_E$-modules to the category of vector bundles on $U$.
By \Cref{lemma-extending-vector-bundles}
$$
\mathrm{Bun}(U)\cong \mathrm{Bun}(\Spec(A_{E,C}))
$$
and thus (by \cite[Theorem 19.5.1.]{scholze2020berkeley} again) it suffices to show that the functor
$$
\widetilde{\omega}\colon \mathrm{Rep}_{\mathcal{O}_E}(\mathcal{G})\to \mathrm{Bun}(\Spec(A_{E,C}))
$$
induced by $\omega$ is exact. As the functor
$$
\mathrm{Bun}(U)\to \mathrm{Bun}(\Spec(A_{E,C})),\ \mathcal{V}\mapsto H^0(U,\mathcal{V})
$$ is left exact it suffices to prove right exactness of $\widetilde{\omega}$. Let $U^\prime$ be the punctured spectrum of $A_{E,C^\prime}$. Then the diagram
$$
\xymatrix{
  \mathrm{Bun}(U)\ar[r]^-{H^0(U,-)} \ar[d]& \mathrm{Bun}(\Spec(A_{E,C})\ar[d] \\
    \mathrm{Bun}(U^\prime)\ar[r]^-{H^0(U^\prime,-)} & \mathrm{Bun}(\Spec(A_{E,C^\prime})
}
$$
commutes because it does when $H^0(U,-)$ resp.\ $H^0(U^\prime,-)$ are replaced by their inverses (which are restriction of vector bundles to $U$ resp.\ $U^\prime$). By the assumption that the base change of $\mathcal{P}$ to $U^\prime$ extends to $\Spec(A_{E,C^\prime})$ we can conclude that the composition
$$
\widetilde{\omega}\otimes_{A_{E,C}} A_{E,C^\prime}\colon \mathrm{Rep}_{\mathcal{O}_E}(\mathcal{G})\to \mathrm{Bun}(\Spec(A_{E,C^\prime}))
$$
is exact.
Let $0\to V_1\to V_2\to V_3\to 0$ be an exact sequence in $\mathrm{Rep}_{\mathcal{O}_E}(\mathcal{G})$ and let $Q$ be the cokernel of $\widetilde{\omega}(V_2)\to \widetilde{\omega}(V_3)$. It suffices to show that $Q=0$. As $Q$ is finitely generated it suffices to show that
$$
Q/\pi Q=Q\otimes_{A_{E,C}}\mathcal{O}_C=0
$$
by Nakayama's lemma. But $\mathcal{O}_C\to \mathcal{O}_{C^\prime}$ is faithfully flat and
$$
(Q\otimes_{A_{E,C}}\mathcal{O}_C)\otimes_{\mathcal{O}_C} \mathcal{O}_{C^\prime}=(Q\otimes_{A_{E,C}}A_{E,C^\prime})\otimes_{A_{E,C^\prime}} \mathcal{O}_{C^\prime}=0
$$
as we know that $Q\otimes_{A_{E,C}}A_{E,C^\prime}=0$.
This finishes the proof.
\end{proof}

Thus from now on we may (and do) assume that $C$ is algebraically closed and even maximally closed as in \Cref{sec:decompleting-a_e}.
By \cite[4.6.20.Remarques]{bruhat_tits_groupes_reductif_sur_un_corps_local_II_schemas_en_groupes} the base change of a parahoric group scheme along an \'etale extension $\mathcal{O}_E\to \mathcal{O}_{E^\prime}$ is again parahoric. The same holds for passing to the completion. Thus from now on we may further assume that $\mathcal{O}_E$ is $\pi$-adically complete and strictly henselian, i.e.\ that $k$ is algebraically closed.
Under these assumptions $E$ is of dimension $1$ (in fact C1, cf.\ \cite{lang_on_quasi_algebraic_closure}) and thus every reductive group over $E$ is automatically quasi-split by Steinberg's theorem (applied to the adjoint quotient).
Moreover, the $\pi$-adic completion of $A_E[1/[\varpi]]$ (for $\varpi\in \mathfrak{m}_C$ non-zero) will be the complete discrete valuation ring
$$
\mathcal{O}_{\mathcal{E}}:=W(C)\hat{\otimes}_{W(k)}\mathcal{O}_E
$$ with algebraically closed residue field $C$. Hence, its fraction field
$$
\mathcal{E}:=W(C)\hat{\otimes}_{W(k)}E
$$
will again be of dimension $1$. This observation, which does not hold over $R_E$, will be crucial as by Steinberg's theorem it implies that every $G$-torsor on $\mathcal{E}$ is trivial (cf.\ \Cref{theorem_steinbergs_theorem}).
As $C$ is algebraically closed the ring $A_E$ is moreover strictly henselian. Hence we can conclude that a $\mathcal{G}$-torsor over $U$ extends to $\Spec(A_E)$ if and only if it is trivial.
Let us start the question on extending torsors by clarifying the assumption that $\mathcal{G}$ is parahoric and not some arbitrary affine smooth model of $G$.

\begin{proposition}
\label{proposition_double_cosets_parahoric}
The double coset space
$$
\mathcal{G}(\mathcal{O}_{\mathcal{E}}) \backslash G(\mathcal{E})/G(A_E[1/\pi])=\{1\}
$$ 
is trivial.
\end{proposition}
\begin{proof}
  The argument in \cite[Proposition 1.4.3. Step 3]{kisin_pappas_integral_models_of_shimura_varieties_with_parahoric_level_structure} works in our situation, however using affine Grassmannians (or in different terminology affine flag varieties) we can give a simpler and more conceptual argument. For this consider the affine Grassmannian $\mathrm{Gr}_{\mathcal{G}}$ of $\mathcal{G}$, i.e., the (\'etale) sheafification of the presheaf
  $$
R\mapsto G(W(R)\hat{\otimes}_{W(k)}E)/\mathcal{G}(W(R)\hat{\otimes}_{W(k)}\mathcal{O}_E)
$$
on the category of perfect $k$-algebras. As $\mathcal{G}$ is parahoric the sheaf $\mathrm{Gr}_{\mathcal{G}}$ is represented by an ind-perfectly proper (strict) ind-scheme (cf.\ \cite[(1.4.2.)]{zhu_affine_grassmannians_and_the_geometric_satake_in_mixed_characteristic} resp. \cite[Corollary 1.3]{richarz_affine_grassmannians_and_geometric_satake_equivalences}). In particular it satisfies the valuative criterion for properness. Moreover, as $C$ and $\mathcal{O}_C$ are strictly henselian
$$
\mathrm{Gr}_\mathcal{G}(C)=G(\mathcal{E})/\mathcal{G}(\mathcal{O}_{\mathcal{E}})
$$
resp.\
$$
\mathrm{Gr}_\mathcal{G}(\mathcal{O}_C)=G(A_E[1/\pi])/\mathcal{G}(A_E).
$$
The claim follows from applying the valuative criterion for properness to $\mathrm{Gr}_{\mathcal{G}}$:
$$
\mathrm{Gr}_\mathcal{G}(C)=\mathrm{Gr}_\mathcal{G}(\mathcal{O}_C).
$$
\end{proof}

From \Cref{proposition_double_cosets_parahoric} we can conclude the following useful criterion.

\begin{corollary}
\label{coro:criterion-extending-torsors-parahoric}
Let $\mathcal{P}$ be a $\mathcal{G}$-torsor over $U$. Then $\mathcal{P}$ extends to $\Spec(A_E)$ if and only if the $G$-torsor 
$$
\mathcal{P}_{|V}
$$ 
over $V=\Spec(A_E[1/\pi])$ is trivial.
Moreover, every $\mathcal{G}$-torsor on $U$ extends to $\Spec(A_E)$ if and only if 
$$
H^1(V,G)=\{1\}.
$$  
\end{corollary}
\begin{proof}
By Beauville-Laszlo glueing (cf.\ \cite{beauville_laszlo_un_lemme_de_descente} resp.\ \Cref{lemma_beauville_laszlo_for_torsors}) there is a bijection of isomorphism classes of $\mathcal{G}$-torsors $\mathcal{P}^\prime$ on $U$ which are trivial on $\Spec(A_E[1/\pi])$ and $\Spec(\mathcal{O}_{\mathcal{E}})$ with the double cosets
$$
\mathcal{G}(\mathcal{O}_{\mathcal{E}}) \backslash G(\mathcal{E})/G(A_E[1/\pi]).
$$ 
But $\mathcal{O}_{\mathcal{E}}$ is strictly henselian which implies that every $\mathcal{G}$-torsor over $\Spec(\mathcal{O}_{\mathcal{E}})$ is trivial as $\mathcal{G}$ is smooth. With \Cref{proposition_double_cosets_parahoric} we can conclude the first assertion. Let us prove the second.
If 
$$
H^1(V, G)=\{1\},
$$ 
then by what we have shown so far, every $\mathcal{G}$-torsor on $U$ extends to $\Spec(A_E)$, i.e.\ is trivial. Conversely, let $\mathcal{P}^\prime$ be a $G$-torsor on $V$. As the field $\mathcal{E}$ is of dimension $1$ the base change of $\mathcal{P}^\prime$ to $\Spec(\mathcal{E})$ is trivial by Steinberg's theorem (cf.\ \Cref{theorem_steinbergs_theorem}). In particular, using Beauville-Laszlo again (cf.\ \Cref{lemma_beauville_laszlo_for_torsors}), we can extend $\mathcal{P}^\prime$ to a $\mathcal{G}$-torsor $\mathcal{P}$ on $U$. By assumption the $\mathcal{G}$-torsor $\mathcal{P}$ extends to $\Spec(A_E)$ and is thus trivial as we assumed that $C$ (and thus its residue field $k^\prime$) is algebraically closed. In particular, $\mathcal{P}^\prime=\mathcal{P}_{|V}$ is trivial. 
This finishes the proof.  
\end{proof}

We now provide an example showing that \Cref{proposition_double_cosets_parahoric} fails in the simplest case if $\mathcal{G}$ is not assumed to be parahoric.

\begin{example}
\label{example:non-trivial-torsors-non-parahoric-group-scheme}
Let $G=\mathbb{G}_{m,E}$ be the multiplicative group and let $\mathcal{G}$ be the smooth model of $G$ over $\mathcal{O}_E$ such that $\mathcal{G}(\mathcal{O}_E)\subseteq \mathcal{O}_E^\times$ is the subgroup of one-units 
$$
\mathcal{G}(\mathcal{O}_E)=\{a\in \mathcal{O}_E^\times\ |\ a\equiv 1 \textrm{ mod }\pi\}
$$ 
(the group scheme $\mathcal{G}$ can be constructed as the dilatation of $\mathbb{G}_{m,\mathcal{O}_E}$ along the unit section of the special fiber). Then
$$
\mathcal{G}(\mathcal{O}_{\mathcal{E}})\backslash G(\mathcal{E})/G(A_{E}[1/\pi])\neq \{1\}.
$$
In fact, 
$$
\mathcal{G}(\mathcal{O}_{\mathcal{E}})\backslash G(\mathcal{E})\cong \pi^\Z\times C^{\times}
$$
and the image of $G(A_E[1/\pi])$ in $\pi^\Z\times C^{\times}$ is given by
$$
\pi^\Z\times \mathcal{O}_{C}^\times.
$$
Thus we obtain a bijection
$$
\mathcal{G}(\mathcal{O}_{\mathcal{E}})\backslash G(\mathcal{E})/G(A_E[1/\pi])\cong C^{\times}/\mathcal{O}_{C}^\times\neq \{1\}.
$$
In particular, we can conclude that there exist $\mathcal{G}$-torsors over $U$ which do not extend to $\Spec(A_E)$ but are trivial on $V$.
\end{example}

The following corollary will be needed.

\begin{corollary}
\label{corollary_h1_trivial_for_split_groups}
Assume that $G$ is split. Then
$$
H^1(V,G)=\{1\}.
$$
\end{corollary}
\begin{proof}
Let $\mathcal{G}$ be the split reductive model of $G$ over $\mathcal{O}_E$.
By \Cref{proposition_torsors_general_split_case} we know that $\mathcal{G}$-torsors on $U$ extend to $\Spec(A_E)$. Using \Cref{coro:criterion-extending-torsors-parahoric} we can conclude.  
\end{proof}

Note that we can conclude that if $G$ is split and $\mathcal{G}$ \textit{any} parahoric model of $G$, then each $\mathcal{G}$-torsor on $U$ extends to $\Spec(A_E)$.

The case $G=\mathrm{PGL}_n$ will be of particular use, when we discuss tori.
Note that \Cref{corollary_h1_trivial_for_split_groups} is wrong for the $2$-dimensional regular local noetherian ring $R_E=\mathcal{O}_E[[z]]$ in this case.

\begin{proposition}
\label{proposition_cohomology_of_tori}
Let $T$ be a torus over $E$. Then
$$
H^1(V,T)=0
$$
and the torsion subgroup 
$$
H^2(V,T)_{\mathrm{tor}}=0
$$
of $H^2(V,T)$ is trivial.
In particular (cf.\ \Cref{coro:criterion-extending-torsors-parahoric}), every $\mathcal{T}^\circ$-torsor over $U$ under the unique parahoric model $\mathcal{T}^\circ$ of $T$ is trivial.
\end{proposition}
\begin{proof}
Let $\overline{E}$ be a separable closure of $E$ and consider the spectral sequence
$$
E^{i,j}_2=H^i(\mathrm{Gal}(\overline{E}/E),H^j(V_{\overline{E}},T_{\overline{E}}))\Rightarrow H^{i+j}(V,T)
$$
with 
$$
V_{\overline{E}}:=V\times_{\Spec(E)}\Spec(\overline{E}).
$$
Let $E^\prime$ over $E$ be a finite separable extension such that the base change $T_{E^\prime}$ of $T$ splits.
Then by \Cref{propositon-locally-free-sheaves-after-inverting-p} (or \Cref{corollary_h1_trivial_for_split_groups})
$$
H^1(V_{E^\prime},T_{E^\prime})=0,
$$
where
$$
V_{E^\prime}:=V\times_{\Spec(E)}\Spec(E^\prime).
$$
Passing to the limit over finite separable extensions $E^\prime$ of $E$ we can deduce
$$
H^1(V_{\overline{E}},T_{\overline{E}})=0.
$$
from \Cref{propositon-locally-free-sheaves-after-inverting-p} (or \Cref{corollary_h1_trivial_for_split_groups}) as for some $E^\prime/E$ finite the torus $T_{E^\prime}$ will be split.
As $E$ is of cohomologial dimension $1$ (thus of strict cohomological dimension $\leq 2$, cf.\ \cite[Chapitre I. Proposition 3.2.13]{serre_cohomologie_galoisienne}) we get
$$
H^i(\mathrm{Gal}(\overline{E}/E),H^j(V_{\overline{E}},T_{\overline{E}}))=0
$$
for $i\geq 3$ and that the group
$$
H^2(\mathrm{Gal}(\overline{E}/E),H^0(V_{\overline{E}},T_{\overline{E}}))
$$
is divisible. To see the second claim, set 
$$
M:=H^0(V_{\overline{E}},T_{\overline{E}}).
$$ 
Then for $n\in \Z\setminus\{0\}$ the exact sequences
$$
0\to M[n]\to M\to M/{M[n]}\to 0 
$$
and
$$
0\to M/M[n]\to M\to M/nM\to 0
$$
with $M[n]\subseteq M$ denoting the $n$-torsion submodule of $M$ furnish that multiplication by $n$ is a surjection on
$$
H^2(\mathrm{Gal}(\overline{E}/E),M)
$$
as it factors through surjections
$$
H^2(\mathrm{Gal}(\overline{E}/E),M)\cong H^2(\mathrm{Gal}(\overline{E}/E),M/M[n])\twoheadrightarrow H^2(\mathrm{Gal}(\overline{E}/E),M).
$$ 
The above spectral sequence thus yields a (split) short exact sequence
$$
0\to H^2(\mathrm{Gal}(\overline{E}/E),M)\to H^2(V,T)\to H^0(\mathrm{Gal}(\overline{E}/E),H^2(V_{\overline{E}},T_{\overline{E}}))\to 0.
$$
Assuming that $H^2(V_{\overline{E}},T_{\overline{E}})$ has only trivial torsion, we can derive that there is an isomorphism 
$$
H^2(\mathrm{Gal}(\overline{E}/E),M)\cong H^2(V,T)_{\mathrm{tor}}
$$
on torsion parts, in particular that $H^2(V,T)_{\mathrm{tor}}$ is divisible. Let $E^\prime/E$ be a finite separable extension splitting $T$. Then the inclusion resp.\ the norm
$$
T\to \mathrm{Res}_{E^\prime/E}(\mathbb{G}_m) \textrm{ , } \mathrm{Res}_{E^\prime/E}(\mathbb{G}_m)\to T
$$ compose to multiplication by $[E^\prime:E]$ on $T$. If $H^2(\Spec(A_{E^\prime}[1/\pi]),\mathbb{G}_m)_{\mathrm{tor}}=0$, then we can conclude (using Shapiro's isomorphism) that $H^2(V,T)_{\mathrm{tor}}$ is annihilated by $[E^\prime:E]$. Being also divisible this implies 
$$
H^2(V,T)_{\mathrm{tor}}=0
$$ 
as desired.
Hence, it suffices to prove that the torsion part
$$
H^2(V_{E^\prime},\mathbb{G}_m)_{\mathrm{tor}}=0
$$
vanishes for every separable algebraic extension $E^\prime/E$. Passing to the limit, we may assume that $E^\prime/E$ is a finite separable extension. As $V_{E^\prime}=\Spec(A_{E^\prime}[1/\pi])$ is affine we can apply Gabber's theorem (cf.\ \cite{hoobler_when_is_brx} or \cite[Corollary 3.1.4.2.]{lieblich_twisted_sheaves_and_the_period_index_problem}\footnote{using noetherian approximation for the general case}) and conclude that each class 
$$
\alpha\in H^2(V_{E^\prime},\mathbb{G}_m)_{\mathrm{tor}}
$$ 
is represented by some Azumaya algebra, i.e., there exists some $n$ such that $\alpha$ lies in the image of
$$
H^1(V_{E^\prime},\mathrm{PGL}_n)\to H^2(V_{E^\prime},\mathbb{G}_m).
$$
Now we can apply \Cref{proposition_torsors_general_split_case} to get that
$$
H^1(V_{E^\prime},\mathrm{PGL}_n)=\{1\}
$$
is trivial, which implies $\alpha=0$.
The proof of the statement for $H^2$ is now finished and we turn to show
$$
H^1(V,T)=0
$$
for every torus $T/E$.
If $T=\mathrm{Res}_{E^\prime/E} (\mathbb{G}_m)$ is an induced torus, then 
$$
H^1(V,T)=H^1(V_{E^\prime},\mathbb{G}_m)=0
$$
by \Cref{propositon-locally-free-sheaves-after-inverting-p}.
In general, let $T$ be an arbitrary torus and chose an exact sequence
$$
0\to T^{\prime\prime}\to T^\prime\xrightarrow{\alpha} T\to 0
$$
of tori with $T^\prime$ induced. Then there exists a morphism $\beta\colon T\to T^\prime$ such that $\alpha\circ\beta=n$ is multiplication by some non-zero $n\in \Z$. In particular, the group 
$$
H^1(V,T)
$$
is torsion as $H^1(V,T^\prime)=0$.
Using that the torsion in 
$$
H^2(V,T^{\prime\prime})
$$ 
vanishes we can thus conclude $H^1(V,T)=0$ from the exact sequence
$$
0=H^1(V,T^\prime)\to H^1(V,T)\to H^2(V,T^{\prime\prime}).
$$
\end{proof}

We record the following vanishing result for multiplicative coefficients.

\begin{lemma}
\label{lemma_second_cohomology_with_finite_coefficients}
For every finite multiplicative group scheme $D/E$ (cf.\ \cite{grothendieck_sga_3_expose_9_groupes_de_type_multiplicative}) the second flat cohomology group 
$$
H^2_{\mathrm{fl}}(V,D)=0
$$
vanishes.
\end{lemma}
\begin{proof}
We may choose a short exact sequence (for the flat topology)
$$
0\to D\to T\to T^\prime\to 0
$$
of multiplicative group schemes with $T$ and $T^\prime$ tori over $E$. Then the statement follows from \Cref{proposition_cohomology_of_tori} by taking the associated long exact sequence in cohomology.
\end{proof}

We record the following corollary of \Cref{proposition_cohomology_of_tori}.

\begin{lemma}
\label{lemma_central_isogenies}
Let $1\to H\to G^\prime\to G\to 1$ be a central extension of two (connected) reductive groups $G^\prime$ resp.\ $G$ over $E$. Then
$$
H^1(V,G^\prime)=1 \Leftrightarrow H^1(V,G)=1. 
$$  
\end{lemma}
\begin{proof}
Let $G^\mathrm{ad}$ be the adjoint quotient of $G$. Then $G^\mathrm{ad}$ is also the adjoint quotient of $G^\prime$. Argueing for the pairs $(G,G^\mathrm{ad})$ resp.\ $(G^\prime,G^\mathrm{ad})$ with the respective central extensions reduces to the case that $G$ is adjoint.
Let $H^\circ\subseteq H$ be the connected component of the identity. By \Cref{proposition_cohomology_of_tori} 
$$
H^1(V,G^\prime)
$$ 
vanishes if and only if
$$
H^1(V,G^\prime/H^\circ)
$$
(noting that the image of the connecting morphism
$$
H^1(V,G^\prime/H^\circ)\to H^2(V,H^\circ)
$$
lands inside the torsion subgroup as each $G^\prime/{H^\circ}$-torsor on $A_{E}$ is trivial after base change along some finite extension of $E$).
Hence we may assume that $H^\circ=\{1\}$ is trivial and thus that $H$ is finite. By \Cref{lemma_second_cohomology_with_finite_coefficients} the group $H^2(\Spec(A_E[1/\pi]),H)$ vanishes. Hence,
$$
H^1(\Spec(A_E[1/\pi]),G^\prime)=\{1\}
$$
implies
$$
H^1(\Spec(A_E[1/\pi]),G)=\{1\}.
$$
Assume conversely that $H^1(\Spec(A_E[1/\pi]),G)=\{1\}$, then it suffices to see that the morphism
\[
  H^1(\Spec(A_E[1/\pi]),H)\to H^1(\Spec(A_E[1/\pi]),G^\prime)
\]
is trivial. But the embedding $H\to G^\prime$ factors through some maximal torus $T\subseteq G^\prime$ and $H^1(\Spec(A_E[1/\pi]),T)=\{1\}$ by \Cref{proposition_cohomology_of_tori}. This finishes the proof.
\end{proof}

We can now turn to our main theorem about extending torsors on the punctured spectrum of $A_E$.

We start by establishing the following criterion for extending a torsor to $\Spec(A_E)$, which is more precise than \Cref{coro:criterion-extending-torsors-parahoric}.

\begin{lemma}
  \label{proposition_criterion_via_crystalline_part}
Let $\mathcal{P}$ be a $\mathcal{G}$-torsor on $U$. Then $\mathcal{P}$ extends to $\Spec(A_E)$ if the restriction $\mathcal{P}_{|U_\crys}$ of $\mathcal{P}$ to the crystalline part $U_\crys\subseteq U$ is trivial (cf.\ \Cref{sec:notations}).
\end{lemma}
\begin{proof}
Note that $G$ is quasi-split by our assumption that $k$ is algebraically closed.By \Cref{coro:criterion-extending-torsors-parahoric} it suffices to prove that $\mathcal{P}_{|V}$ is trivial. Let 
$$
T\subseteq B\subseteq G
$$ 
be a maximal torus and a Borel. 
As it is trivial the torsor
$$
\mathcal{P}_{|U_\crys}
$$
admits a reduction to $B$ over $U_\crys$. 
By \Cref{lemma_spectrum_of_a_inf} and \Cref{corollary_local_rings_at_distinguished_elements_are_dvr} for $s\in V\setminus U_\crys$ the local ring
$$
\mathcal{O}_{V,s}
$$ 
is a discrete valuation ring. Hence by properness of the quotient $\mathcal{P}_{|V}/B$ over $E$ and the valuative criterion for properness the torsor $\mathcal{P}_{|V}$ admits a reduction 
$$
\mathcal{P}^\prime\in H^1(V,B)
$$ 
to $B$ over the whole of $V$. In other words, there exists a $B$-torsor $\mathcal{P}^\prime$ over $V$ such that
$$
\mathcal{P}^\prime\times^B G\cong \mathcal{P}_{|V}.
$$
Let $\mathrm{rad}(B)\subseteq B$ be the unipotent radical of $B$ and consider the natural map
$$
H^1(V,B)\xrightarrow{\Phi} H^1(V,B/\mathrm{rad}(B)). 
$$
The fiber $\Phi^{-1}(\Phi(\mathcal{P}^\prime))$ containing $\mathcal{P}^\prime\in H^1(V,B)$ can naturally be identified with the set
$$
H^1(V,\mathrm{rad}(B)_{\mathcal{P}^\prime}),
$$ 
where 
$$
\mathrm{rad}(B)_{\mathcal{P}^\prime}:=(\mathrm{rad}(B)\times_{\Spec(E)}V)\times^B \mathcal{P}^\prime
$$
is a twisted form of the constant group scheme $\mathrm{rad}(B)\times_{\Spec(E)}V$ over $V$. Here $B$ acts on $\mathrm{rad}(B)$ via conjugation. As $\mathrm{rad}(B)$ admits a canonical, i.e., $B$-stable, filtration whose graded pieces are vector spaces over $E$ (with $B$ acting linearly) the unipotent group scheme $\mathrm{rad}(B)_{\mathcal{P}^\prime}$ over $V$ admits a filtration with graded pieces vector bundles over $V$. As $V$ is affine the (\'etale) cohomology with coefficients in quasi-coherent sheaves, in particular vector bundles, vanishes and therefore 
$$
H^1(V,\mathrm{rad}(B)_{\mathcal{P}^\prime})=\{1\}
$$
as well.
In particular, the map $\Phi$ is injective. 
By \Cref{proposition_cohomology_of_tori} the pointed set
$$
H^1(V,B/\mathrm{rad}(B))\cong H^1(V,T)=\{1\}
$$
is trivial and by injectivity of $\Phi$ we can conclude that $\mathcal{P}^\prime$, hence $\mathcal{P}_{|V}$, is trivial.
Thus we the proof is finished.
.
\end{proof}

Now we can prove our main theorem for $A_E$, cf.\ \Cref{theorem_introduction_main_theorem_a}.

\begin{theorem}
  \label{sec:extend-tors-mathrmsp-1-main-theorem-a-e}
  Let $\mathcal{G}$ be a parahoric group scheme over $\mathcal{O}_E$. Then every $\mathcal{G}$-torsor over the punctured spectrum $U=\Spec(A_E)\setminus\{s\}\subseteq \Spec(A_E)$ extends to $\Spec(A_E)$.
\end{theorem}
\begin{proof}
  Let $G:=\mathcal{G}\otimes_{\mathcal{O}_E}E$, and set as before
  \[
    V:=\Spec(A_E[1/\pi]),
  \]
  \[
    U_{\mathrm{crys}}=\Spec(A_{E,\mathfrak{p}_\crys}).
  \]
  By \Cref{proposition_criterion_via_crystalline_part} it suffices to show that the map
  \[
    H^1(V,G)\to H^1(U_\crys,G)
  \]
  is trivial, or equivalently that $H^1(V,G)=0$. By \Cref{lemma_central_isogenies} we are free to replace $G$ by a central extension. In particular, we may assume that the derived subgroup $G^{\mathrm{der}}$ of $G$ is simply connected and $G\cong G^{\mathrm{der}}\times G/G^{\mathrm{der}}$. By \Cref{proposition_cohomology_of_tori} we can then reduce to the case that $G$ is semisimple (and simply connected). Then $G$ is Weil restriction and after replacing $E$ by a suitable finite extension we can then assume that $G$ is the product of (quasi-split) simple, simply connected groups. As a quasi-split group of type $E_8$ is automatically split, we may by \Cref{corollary_h1_trivial_for_split_groups}, then assume that $G$ is quasi-split, simple and not of type $E_8$.
  Let
  \[
    \tilde{A}_E
  \]
  be the $\mathcal{O}_E$-algebra introduced in \Cref{sec:decompleting-a_e}.
  Set
  \[
    \tilde{V}:=\Spec(\tilde{A}_E[1/\pi]).
  \]
  By \Cref{sec:decompleting-a_e-5-torsors-on-tilde-a-e-1-over-pi} the canonical morphism
  \[
    H^1(\tilde{V},G)\to H^1(V,G)
  \]
  is a bijection.
  Let
  \[
    \tilde{U}_{\crys}\subseteq \tilde{V}
  \]
  be the ``crystalline'' part of $\tilde{V}$, cf.\ \Cref{sec:decompleting-a_e-4-crystalline-part-of-tilde-a-e}. By \Cref{sec:decompleting-a_e-4-crystalline-part-of-tilde-a-e}
  \[
    U_\crys\subseteq \tilde{U}_{\crys}\times_{\tilde{V}} V,
  \]
  and
  \[
    \tilde{U}_\crys=\Spec(\tilde{A}_{E,\tilde{\mathfrak{p}}_\crys})
  \]
  is the spectrum of a valuation ring (with fraction field $\tilde{K}_E=\mathrm{Frac}(\tilde{A}_E)$).
  It suffices to show that each $G$-torsor $\tilde{\mathcal{P}}$ on $\tilde{U}_\crys$ is trivial.
By \Cref{sec:decompleting-a_e-1-cohomological-dimension-of-tilde-a-e}, \Cref{sec:decompleting-a_e-1-kato-cohomologie-of-tilde-k-e} and our reductions on $G$, all assumptions for \cite[Th\'eor\`eme 4]{gille_cohomologie_galoisienne_des_groupes_quasi_deployes_sur_des_corps_de_dimension_cohomologique_2} are met and we can conclude that the torsor
  \[
    \tilde{\mathcal{P}}_{|\tilde{K}_E}
  \]
  is trivial. As
  \[
    \tilde{A}_{E,\tilde{\mathfrak{p}}_\crys}
  \]
  is a valuation ring and $G$ quasi-split, we can then (similar to the proof of \Cref{proposition_criterion_via_crystalline_part}) conclude that $\tilde{\mathcal{P}}$ admits a reduction to a maximal torus $T$ of $G$ (defined over $E$).
  Thus, we may replace $G$ by $T$. Note that $T$ is an induced torus over $E$ as we assumed that $G$ is simply connected. Thus,
  \[
    T\cong \prod\limits_{i=1}^n\mathrm{Res}_{E_i/E}\Gm
  \]
  for some finite separable extensions $E_i$ of $E$. We then get that
  \[
    H^1(\tilde{U}_{\crys},T)=0
  \]
  because the rings
  \[
    \tilde{A}_{E,\tilde{\mathfrak{p}}_\crys}\otimes_{E} E_i=\tilde{A}_{E_i,\tilde{\mathfrak{p}}_\crys}
  \]
  for $i=1,\ldots, n$ are local. This finishes the proof of our main theorem.
\end{proof}

For clarity, let us highlight at several steps in the proof of \Cref{sec:extend-tors-mathrmsp-1-main-theorem-a-e} which features of the rings $\tilde{A}_E, A_E$ we used. Roughly, arguments related to the ``crystalline part'' had to be carried out over $\tilde{A}_E$, while arguments related to the ``non-crystalline'' part over $A_E$. More precisely, in \Cref{proposition_double_cosets_parahoric} we used the (Witt vector) affine Grassmannian and it is not clear whether this results holds over $\tilde{A}_E$. Unfortunately, the results of Fargues-Fontaine, which are used in \Cref{section_spectrum_of_a}, are not known over $\tilde{A}_E$. In particular, we don't know a concrete description of the ``non-crystalline'' part of $\Spec(\tilde{A}_E)$ and thus cannot circumvent a statement like \Cref{proposition_criterion_via_crystalline_part}, which argues on $A_E$. However, for proving triviality on the crystalline part, we have to use $\tilde{A}_E$, because only there we can prove the generic triviality of torsors. Namely, we don't know if the fraction field of $K_E$ has cohomological dimension $\leq 2$, cf.\ \Cref{sec:decompleting-a_e-1-cohomological-dimension-of-tilde-a-e}.
\section{Extending torsors to $\Spec(R_E)$}
\label{section_torsors_re}

We now turn to the question of extending torsors on the punctured spectrum of (some) regular 2-dimensional local noetherian rings.
We continue to use the notation from \Cref{sec:notations}. Thus $E$ denotes a complete discretely valued field, $\mathcal{O}_E$ its ring of integers, etc.
Furthermore we let $R_E$ be given by
$$
R_E:=\mathcal{O}_E[[z]].
$$
We again denote by
$$
s:=s_{R_E}\in \Spec(R_E)
$$
the unique closed point and by
$$
U:=U_{R_E}:=\Spec(R_E)\setminus\{s\}
$$
its complement.
Moreover, set
$$
V:=V_{R_E}:=\Spec(R_E[1/\pi]).
$$
If confusion with our previous notation for $A_E$ from \Cref{sec:notations} may be possible we will add subscripts.
First of all let us recall the well-known fact that vector bundles on $U$ extend uniquely to $\Spec(R_E)$.

\begin{lemma}
\label{lemma_vector_bundles_extend_noetherian_case}
For vector bundles the restriction functor
$$
\mathrm{Bun}(\Spec(R_E))\to \mathrm{Bun}(U)
$$
is an equivalence.
\end{lemma}
\begin{proof}
  Fully faithfullness follows from $H^0(U,\mathcal{O}_U)=R_E$ which is implied by normality of $R_E$. Conversely let $\mathcal{V}$ be a vector bundle on $U$ and let $M:=H^0(U,\mathcal{V})$. It is enough to prove that $M$ is finite projective over $R_E$ as then the vector bundle over $\Spec(R_E)$ associated with $M$ will extend $\mathcal{V}$. By the Auslander-Buchsbaum formula
 $$
\mathrm{pd}(M)+\mathrm{depth}(M)=2
$$
where $\mathrm{pd}(M)$ and $\mathrm{depth}(M)$ are the projective dimension and depth of $M$. Hence, it suffices to proof $\mathrm{depth}(M)=2$. For this it suffices to proof that $M/\pi$ is torsion-free over $R_E/\pi$. But applying cohomology to the exact sequence
$$
0\to \mathcal{V}\overset{\pi}{\to}\mathcal{V}\to \mathcal{V}/\pi\to 0 
$$
we see that $M/\pi$ embeds into $H^0(U,\mathcal{V}/\pi)$ which is torsionfree over $R_E/\pi$ as $\mathcal{V}/\pi$ is a vector bundle on $\Spec(R_E/\pi)\cap U=\Spec(\mathrm{Frac}(R_E/\pi))$.
\end{proof}

Hence, we can apply our general results from \Cref{section_generalities_on_extending_torsors}, in particular \Cref{lemma-descent-criterion-for-extending}.
For this let us define a morphism of $\mathcal{O}_E$-algebras
$$
f\colon R_E\to A_E.
$$
Namely, let $\varpi\in \mathfrak{m}_C\setminus\{0\}$ be an arbitrary element and define $f$ by
$$
f(z):=[\varpi].
$$
\begin{lemma}
\label{lemma_varphi_faithfully_flat}
The morphism $f\colon R_E\to A_E$ is faithfully flat.
\end{lemma}
\begin{proof}
The proof in \cite[Lemma 4.30]{bhatt_morrow_scholze_integral_p_adic_hodge_theory} works in this situation as well.
\end{proof}

Thus we are able to apply \Cref{lemma-descent-criterion-for-extending} to obtain our main result for $R_E$, cf.\ \Cref{sec:introduction-corollary-main-theorem}.

\begin{proposition}
\label{proposition_criterion_for_extending_noetherian_case}
For every parahoric group scheme $\mathcal{G}$ over $\mathcal{O}_E$ a $\mathcal{G}$-torsor $\mathcal{P}$ on $U$ extends to $\Spec(R_E)$.
\end{proposition}
\begin{proof}
The statement follows by descent (cf.\ \Cref{lemma_varphi_faithfully_flat} and \Cref{lemma-descent-criterion-for-extending}) from our main theorem \Cref{sec:extend-tors-mathrmsp-1-main-theorem-a-e} over $A_E$.  
\end{proof}

This result extends \cite[Proposition 1.4.3.]{Kisin2015a} to complete generality.

\section{ $v$-local triviality for perfect rings }
\label{sec:v-stack-g}

We keep the notations from \Cref{sec:notations}. In particular, $k$ is a perfect field of $\mathrm{char}(k)=p>0$, and let $G/E$ a reductive group. In this section we want to prove \Cref{sec:v-stack-g-2-torsors-trivial-in-v-topology} by deducing it from our main theorem \Cref{sec:extend-tors-mathrmsp-1-main-theorem-a-e}.

We start with some lemmata.
Let $\mathrm{Alg}_k^{\mathrm{perf}}$ denote the category of perfect $k$-algebras.\footnote{We don't use ``$\mathrm{Perf}_k$'' as the latter is sometimes used for the category of perfectoid spaces over $\Spa(k)$.}
For each $R\in \mathrm{Alg}_k^{\mathrm{perf}}$ we set
\[
  A(R):=A_E(R)[1/\pi]:=W(R)\hat{\otimes}_{W(k)}\mathcal{O}_E,
\]
i.e., $A_E(R)=\mathcal{O}_E\otimes_{W(k)}W(R)$ if $\mathrm{char}(E)=0$, or $A_E=R((\pi))$ if $E=k((\pi))$.

We define the functor
\[
  F_G\colon \mathrm{Alg}_k^{\mathrm{perf}}\to (\Sets),\ R\mapsto H^1(\Spec(A(R)[1/\pi])_\et,G).
\]

\begin{lemma}
  \label{sec:v-stack-g-3-torsors-on-witt-vectors-commute-with-colimits}
  The functor $F_G\colon  \mathrm{Alg}_k^{\mathrm{perf}}\to (\Sets)$ commutes with filtered colimits and finite products.
\end{lemma}
\begin{proof}
  The case of finite products is clear. Let $R=\varinjlim\limits_{i\in I}R_i$ be a filtered colimit in $\mathrm{Alg}^{\mathrm{perf}}_k$. Then
  \[
    \widetilde{A(R)}:=\varinjlim\limits_{i\in I}A(R_i)
  \]
  is $\pi$-torsion free, henselian along $(\pi)$, and has the perfect ring $R$ as a quotient modulo $(\pi)$. Hence, its $\pi$-completion is $A(R)$.
  By \cite[Theorem 5.8.14]{gabber_ramero_almost_ring_theory} we obtain
  \[
    F_G(R)=H^1(\Spec(A(R)[1/\pi])_\et,G)\cong H^1(\Spec(\widetilde{A(R)}[1/\pi])_\et,G).
  \]
  Using that $G$ is affine and smooth, the last term identifies with
  \[
    \varinjlim\limits_{i\in I} F_G(R_i)
  \]
  and we can conclude.
\end{proof}

In the following we will use the $v$-and arc-topology on $\mathrm{Alg}^{\mathrm{perf}}_k$, which were introduced in \cite{bhatt_scholze_projectivity_of_the_witt_vector_affine_grassmannian} and \cite{bhatt2018arc}. 

We recall that a normal integral domain $R$ (e.g., a valuation ring) is absolutely integrally closed, i.e., each monic monomial in $R[x]$ has a solution in $R$, if and only if its fraction field is algebraically closed, cf.\ \cite[Definition 3.22]{bhatt2018arc}.

\begin{lemma}
  \label{sec:v-stack-g-4-isomorphism-for-v-sheaves}
  Let $F_1,F_2\colon \mathrm{Alg}^{\mathrm{perf}}_k\to (\Sets)$ be two $v$-sheaves commuting with filtered colimits and finite products. Let $\eta\colon F_1\to F_2$ be a natural transformation. If $\eta_R$ is a bijection for any absolutely integrally closed valuation ring $R\in \mathrm{Alg}^{\mathrm{perf}}_k$ of finite rank, then $\eta$ is an equivalence. 
\end{lemma}
A similar statement for functors into the $\infty$-category $D(\Lambda)^{\geq 0}$ of some ring $\Lambda$ appears in \cite[Proposition 3.28]{bhatt2018arc}, and we can follow the proof there.
\begin{proof}
  By \cite[Lemma 6.2.]{bhatt_scholze_projectivity_of_the_witt_vector_affine_grassmannian} resp.\ \cite[Lemma 3.24]{bhatt2018arc} each $R\in \mathrm{Alg}^{\mathrm{perf}}_k$ has a $v$-cover
  \[
    R\to R^\prime
  \]
  with $R^\prime$ a product of absolutely integrally closed valuation rings (such $R^\prime$ are automatically perfect). Hence, it suffices to check that $\eta_{R}$ is an equivalence for $R$ a product of absolutely integrally closed valuation rings. Using that ultraproducts of absolutely integrally closed valuation rings are again absolutely integrally closed valuation rings, cf.\ \cite[Lemma 3.23]{bhatt2018arc} resp.\ \cite[Lemma 6.2.]{bhatt_scholze_projectivity_of_the_witt_vector_affine_grassmannian} the arguments of \cite[Corollary 3.15]{bhatt2018arc} reduce us the the case that $R$ is an absolutely integrally closed valuation ring. By writing $R$ as a colimit of finitely generated $k$-algebras and using that $F_1,F_2$ commute with filtered colimits, we can then further reduce to the case that $R$ is of finite rank, cf.\ \cite[Lemma 6.4.]{bhatt_scholze_projectivity_of_the_witt_vector_affine_grassmannian} for more details on this reduction. We are using implicitly that the rank of a valuation ring over $k$ does not change in algebraic extensions, e.g., perfections or the passage to an algebraic closure. This finishes the proof.
\end{proof}

We set
\[
  \mathcal{F}_G
\]
as the (2-)functor sending $R\in \mathrm{Alg}^{\mathrm{perf}}_k$ to the groupoid of $G$-torsors on $A(R)[1/\pi]$. 

\begin{lemma}
  \label{sec:v-local-triviality-excision-for-torsors-for-valuation-rings}
  The (2-)functor $\mathcal{F}_G$ is a stack for the arc-topology, in particular for any valuation ring $R\in \mathrm{Alg}^{\mathrm{perf}}_k,$ and any prime ideal $\mathfrak{p}$ with residue field $k(\mathfrak{p})$ the diagram
  \[
    \xymatrix{
      \mathcal{F}_G(R)\ar[r]\ar[d] & \mathcal{F}_G(R/\mathfrak{p}) \ar[d] \\
      \mathcal{F}_G(R_{\mathfrak{p}}) \ar[r] & \mathcal{F}_G(k(\mathfrak{p}))
    }
  \]
  is (homotopy) cartesian.
\end{lemma}
\begin{proof}
  By \cite[Proposition 5.10]{ivanov2020ind} the (2-)functor sending $R\in \mathrm{Alg}_k^{\mathrm{perf}}$ to the category of vector bundles on $\Spec(A(R)[1/\pi])$ is a stack for the arc-topology. In particular, we get the claim for $\mathrm{GL}_n, n\geq 1$. To pass from the case of $\mathrm{GL}_n$ to general $G$ we use the Tannkian description of $G$-torsors on $A(R)[1/\pi])$ as exact tensor functors
  \[
    \mathrm{Rep}_E(G)\to \Bun(A(R)[1/\pi]).
  \]
  Let $R\to R^\prime$ be an arc-cover.
  It suffices to see that a sequence
  \[
    0\to M_1\to M_2\to M_3\to 0
  \]
  of finite projective $A(R)[1/\pi]$ is exact if and only if its base change to $A(R^\prime)[1/\pi]$ is. The ``only if'' direction is clear. Thus, assume exactness of the base change to $A(R^\prime)[1/\pi]$. Then the claim follows as all maximal ideals in $A(R)[1/\pi]$ lie in the image of
  \[
    \Spec(A(R^\prime)[1/\pi])\to \Spec(A(R)[1/\pi]),
  \]
  cf.\ \cite[Corollary 5.6]{ivanov2020ind}. Namely, let $Q$ be the cokernel of $M_2\to M_3$. Then $Q=0$ as it is finitely generated and vanishes at each maximal ideal of $A(R)[1/\pi]$ (because this happens over $A(R^\prime)[1/\pi]$ and the aforementioned surjectivity). In particular, we can write $M_2=M_2^\prime\oplus M_3$ for some finite projective $A(R)[1/\pi]$-module $M_2^\prime$. Then continue the argument with the cokernel $Q^\prime$ of the composition $M_1\to M_2\to M^\prime_2$ etc. This finishes the proof that $\mathcal{F}_G$ is an arc-sheaf.  
      The assertion on valuation rings follows then from \cite[Corollary 2.9]{bhatt2018arc} resp.\ the argument of \cite[Proposition 4.2]{bhatt2018arc}.
\end{proof}

\begin{theorem}
  \label{sec:v-stack-g-2-torsors-trivial-in-v-topology} Let $R$ be a perfect $k$-algebra. For each $G$-torsor $\mathcal{P}$ on $A(R)[1/\pi]$ there exists a $v$-cover $R\to R^\prime$ such that $\mathcal{P}$ is trivial after base change to $A(R^\prime)[1/\pi]$.  
\end{theorem}
\begin{proof}
  By \cite[Lemma 2.12]{bhatt_scholze_projectivity_of_the_witt_vector_affine_grassmannian} each $v$-cover in $\mathrm{Alg}_k^{\mathrm{perf}}$ can be written as a filtered colimit of perfectly finitely presented $v$-covers. These are dominated by the perfection of the composition of a quasi-compact open immersion and a finitely presented proper surjection by \cite[Theorem 2.4]{bhatt_scholze_projectivity_of_the_witt_vector_affine_grassmannian} resp.\ \cite[Theorem 3.12]{rydh2010submersions}, and hence are defined over some finite stage of a filtered colimit. This implies that the $v$-sheafification $F_1$ \footnote{We ignore the set-theoretic issues here. These can be resolved by fixing a sufficiently large cut-off cardinal, cf.\ \cite[Chapter 4]{scholze_etale_cohomology_of_diamonds}.} of the presheaf
  \[
    R\mapsto F_G(R)=H^1(\Spec(A(R)[1/\pi])_\et,G)
  \]
  still commutes with filtered colimits (and finite products), cf.\ \Cref{sec:v-stack-g-3-torsors-on-witt-vectors-commute-with-colimits}.
  In particular, by \Cref{sec:v-stack-g-4-isomorphism-for-v-sheaves} (applied to $F_1$ and the terminal $v$-sheaf $F_2$) it suffices to show the statement for an absolutely integrally closed valuation ring of finite rank $R$. If $R$ is a field, we are done by Steinberg's theorem, cf.\ \Cref{theorem_steinbergs_theorem}, as then $A(R)[1/\pi]$ is a field of cohomological dimension $1$. Thus we may assume that $R$ is not a field. Lifting completions on residue fields, we may assume that for every immediate specialization $\mathfrak{q}\rightsquigarrow \mathfrak{p}$, the valuation ring $R_{\mathfrak{p}}/\mathfrak{q}\subseteq k(\mathfrak{q})$ of rank $1$ is complete. Note that $R=C^+$ is then (a finite rank) open and bounded valuation ring $C^+$ inside a non-archimedean field, algebraically closed field $C=\mathrm{Frac}(R)$. Under all the above assumptions on $R$, we claim that every $G$-torsor on
  \[
    A(R)[1/\pi]
  \]
  is trivial. Indeed, using induction (with the case of rank $1$ settled by \Cref{sec:extend-tors-mathrmsp-1-main-theorem-a-e}) and \Cref{sec:v-local-triviality-excision-for-torsors-for-valuation-rings} it suffices to show that for the height $1$ prime ideal $\mathfrak{p}\subseteq R$ the homotopy pullback
  \[
    \mathcal{F}_G(R_{\mathfrak{p}})\times_{\mathcal{F}_G(k(\mathfrak{p})} \mathcal{F}_G(R/\mathfrak{p})
  \]
  is connected, i.e., this groupoid has one isomorphism class. The induction implies that each of the groupoids
  \[
    \mathcal{F}_G(R_{\mathfrak{p}}), \mathcal{F}_G(k(\mathfrak{p})), \mathcal{F}_G(R/\mathfrak{p})
  \]
  has one isomorphism class.
  The morphism \[
    R_{\mathfrak{p}}\to k(\mathfrak{p})
  \]
  has a section, which implies that the morphism
  \[
    A(R)[1/\pi]\to A(k(\mathfrak{p})[1/\pi]
  \]
  has a section, too (though not as $A(R)[1/\pi])$-algebras, only as $E$-algebras).
  We can conclude that the morphism
  \[
    G(A(R)[1/\pi])\to G(A(k(\mathfrak{p}))[1/\pi])
  \]
  is surjective, which implies the desired connectedness. This finishes the proof. 
\end{proof}

For a scheme $X$ over $E$ let
\[
  LX\colon \mathrm{Alg}^{\mathrm{perf}}_k\to (\Sets),\ R\mapsto X(A(R)[1/\pi])
\]
be its loop functor. If $X$ is quasi-projective, then $LX$ is an arc-sheaf by \cite[Theorem 5.1]{ivanov2020ind}. We can record the following corollary, which generalizes \cite[Corollary 6.4]{ivanov2020ind} to complete generality.

\begin{corollary}
  \label{sec:v-local-triviality-exact-sequence-of-loop-functors}
  Let $P\subseteq G$ be a parabolic subgroup. Then the sequence
  \[
    1\to LP\to LG\to L(G/P)\to 1
  \]
  of arc-sheaves is exact in the $v$-topology.
\end{corollary}
\begin{proof}
  Left-exactness is clear and right exactness follows from \ref{sec:v-local-triviality-exact-sequence-of-loop-functors}. Namely, for any $R\in \mathrm{Alg}^{\mathrm{perf}}_k$ we have an exact sequence
  \[
    1\to LP(R)\to LG(R)\to L(G/P)(R)\to H^1(\Spec(A(R)[1/\pi]
    )_\et,P)
  \]
  of pointed sets and by affineness of $\Spec(A(R)[1/\pi])$
  \[
    H^1(\Spec(A(R)[1/\pi])\et,P)\cong H^1(\Spec(A(R)[1/\pi])_\et,M),
  \]
  where $M=P/\mathrm{Rad}_{u}P$ is the Levi quotient of $P$, cf.\ the proof of \ref{proposition_criterion_via_crystalline_part}. Now, it suffices to apply \ref{sec:v-stack-g-2-torsors-trivial-in-v-topology} to $M$.
\end{proof}

Let us record another corollary from (the proofs of) \Cref{sec:v-stack-g-2-torsors-trivial-in-v-topology} and \Cref{proposition_double_cosets_parahoric}. Recall that $\mathcal{G}/\mathcal{O}_E$ is a parahoric model of $G$.

\begin{corollary}
  \label{sec:v-local-triviality-complete-valuation-ring}
  Let $C$ be a perfect non-archimedean field of characteristic $p$, and let $C^+\subseteq C$ be an open and bounded valuation ring. Let $\varpi\in C$ be a pseudo-uniformizer. Then each $\mathcal{G}$-torsor on $\Spec(A(C^+))\setminus \{\pi,[\varpi]\}$ extends to a $\mathcal{G}$-torsor on $\Spec(A(C^+))$.
\end{corollary}
\begin{proof}
  By \cite[Theorem 2.7]{kedlaya2020some} resp.\ \cite[Proposition 14.2.6]{scholze2020berkeley} the statement holds for $\mathrm{GL}_n, n\geq 1$. Thus, the proof of \Cref{lemma_descent_along_extension_of_c} works for $\mathcal{O}_C$ replaced by $C^+$ again, and we can argue by descent in $(C,C^+)$.  In particular, we may assume that $C$ is algebraically closed. Then $A(C^+)$ is strictly henselian along its maximal ideal.
  The proof of \Cref{proposition_double_cosets_parahoric} via the valuative criterion for properness works for $C^+$ again\footnote{We thank Ian Gleason for pointing this out to us.}. As in \Cref{coro:criterion-extending-torsors-parahoric} we see therefore that we can equivalently prove that each $G$-torsor on $A(C^+)[1/\pi]$ is trivial. By \cite[Theorem 5.8.14]{gabber2003almost} resp.\ \ref{sec:v-stack-g-3-torsors-on-witt-vectors-commute-with-colimits} this statement commutes with filtered colimits in $(C,C^+)$. In particular, we may assume that $C^+$ is of finite rank (and still $C$ algebraically closed).
  Now we argue via descent again, i.e., use the analogs of \Cref{coro:criterion-extending-torsors-parahoric} and \Cref{lemma_descent_along_extension_of_c}, and reduce to the case that for each immediate specialization $\mathfrak{q}\rightsquigarrow \mathfrak{p}$ of the finite rank valuation ring $C^+$ the rank $1$ valuation ring
  \[
    (C^+/\mathfrak{q})_{\mathfrak{p}}\subseteq k(\mathfrak{q})
  \]
  is complete. Then by the proof of \Cref{sec:v-stack-g-2-torsors-trivial-in-v-topology} we see that each $G$-torsor on $A(C^+)[1/\pi]$ is trivial as desired.
 \end{proof}

\section{A specialization map between mixed-characteristic affine Grassmannians}
\label{section_specialization_map}

Let $k$ be an algebraically closed field of char $p>0$ and let $C/W(k)[1/p]$ be an algebraically closed, non-archimedean field with residue field $k^\prime$. After possibly enlarging $k$ we may without losing generality assume $k=k^\prime$.
In this section we want to use \Cref{sec:extend-tors-mathrmsp-1-main-theorem-a-e} to concoct for a parahoric group scheme $\mathcal{G}$ over $W(k)$ a canonical specialization map
$$
\mathrm{sp}\colon \mathrm{Gr}_{\mathcal{G}}^{B_\dR^+}(C)\to \mathrm{Gr}_\mathcal{G}^{W}(k) 
$$
between the mixed characteristic affine Grassmannians $\mathrm{Gr}_{\mathcal{G}}^{B_\dR^+}$ and $\mathrm{Gr}_{\mathcal{G}}^{W}$.
The existence of the specialization map is motivated by results of Richarz \cite[Theorem 1.19]{richarz_affine_grassmannians_and_geometric_satake_equivalences} and the definition of a mixed-characteristic Beilinson-Drinfeld Grassmannian  \cite[Definition 20.3.1.]{scholze2020berkeley}. 
In fact, using \Cref{sec:extend-tors-mathrmsp-1-main-theorem-a-e} it is in fact possible to prove that the mixed-characteristic Beilinson-Drinfeld Grassmannian is ind-proper (cf.\ \cite[Section 21.2]{scholze2020berkeley}).

Let us recall the definition of both affine Grassmannians (we content ourselves with their $k$ resp.\ $C$-valued points.\footnote{For the precise geometric structure as a $v$-sheaf we confer to \cite[Section 20.3]{scholze2020berkeley}.}
We will denote by $C^\flat$ the tilt of $C$.
By definition (cf.\ \cite{zhu_affine_grassmannians_and_the_geometric_satake_in_mixed_characteristic} or \cite{bhatt_scholze_projectivity_of_the_witt_vector_affine_grassmannian}) the $k$-valued points of the Witt vector affine Grassmannian for $\mathcal{G}$ (or better Witt vector affine flag variety) are pairs
$$
(\mathcal{P},\alpha)
$$ 
with $\mathcal{P}$ a $\mathcal{G}$-torsor on $\Spec(W(k))$ and $\alpha$ a trivialization of $\mathcal{P}_{|\Spec(W(k)[1/p])}$.
On the other hand, the $C$-valued points of the $B^+_\dR$-affine Grassmannian (cf.\ \cite[Definition 19.1.1.]{scholze2020berkeley}) are pairs
$$
(\mathcal{P}^\prime,\alpha^\prime)
$$  
with $\mathcal{P}^\prime$ a $\mathcal{G}$-torsor on $\Spec(B^+_{\dR}(C))$ and $\alpha^\prime$ a trivialization of $\mathcal{P}_{|\Spec(B_\dR(C)}$.
We note that, as $B^+_\dR$ contains an algebraic closure $\overline{W(k)[1/p]}$ of $W(k)[1/p]$, a $\mathcal{G}$-torsor over $\Spec(B^+_{\dR}(C))$ is just a torsor under the split reductive geometric generic fiber $\mathcal{G}_{\overline{W(k)[1/p]}}$ of $\mathcal{G}$.
 
Now, let us construct the specialization map
$$
\mathrm{sp}\colon \mathrm{Gr}^{B^+_\dR}_{\mathcal{G}}(C)\to \mathrm{Gr}^W_{\mathcal{G}}(k).
$$
Let $(\mathcal{P}^\prime,\alpha^\prime)\in \mathrm{Gr}^{B^+_\dR}_{\mathcal{G}}(C)$ be given. 
The kernel of Fontaine's map 
$$
\theta\colon A_{\inf}=W({\mathcal{O}_{C^\flat}})\to \mathcal{O}_C
$$
is generated by a non-zero divisor $\xi$. In fact, we may simply take $\xi$ of the form
$$
\xi=p-[\varpi]
$$
for a suitable $\varpi\in \mathfrak{m}_{C^\flat}$.
The $\xi$-adic completion of $A_{\inf}[1/p]$ is by definition Fontaine's ring 
$$
B^+_\dR.
$$
Using the Beauville-Laszlo gluing lemma (cf.\ \cite{beauville_laszlo_un_lemme_de_descente}) and the given data $(\mathcal{P}^\prime,\alpha^\prime)$ we can modify the trivial $\mathcal{G}$-torsor $\mathcal{P}_0$ on
$$
\Spec(A_{\inf})\setminus\{s\}
$$
at the point $\infty\in \Spec(A_{\inf})\setminus\{s\}$ defined by $\xi$ (cf.\ \Cref{lemma_spectrum_of_a_inf}). Thus we obtain canonically a $\mathcal{G}$-torsor $\mathcal{P}_1$ over $\Spec(A_{\inf})\setminus\{s\}$ with an isomorphism 
$$
{\mathcal{P}_1}_{|\Spec(A_{\inf})\setminus\{s,\infty\}}\cong {\mathcal{P}_0}_{|\Spec(A_{\inf})\setminus\{s,\infty\}}.
$$
In particular, the torsor $\mathcal{P}_1$ is trivial when restricted to the crystalline part $U_\crys\subseteq \Spec(A_{\inf})$ (cf.\ \Cref{lemma_spectrum_of_a_inf}). By \Cref{sec:extend-tors-mathrmsp-1-main-theorem-a-e} (or the weaker version \Cref{proposition_criterion_via_crystalline_part}) the $\mathcal{G}$-torsor $\mathcal{P}_1$ extends uniquely to a $\mathcal{G}$-torsor $\mathcal{P}_2$ on $\Spec(A_{\inf})$.
In particular, we still have a canonical trivialization
$$
\mathcal{P}_{2|\Spec(A_{\inf})\setminus\{s,\infty\}}\cong \mathcal{P}_{1|\Spec(A_{\inf})\setminus\{s,\infty\}} \cong \mathcal{P}_{0|\Spec(A_{\inf})\setminus\{s,\infty\}}
$$ 
of $\mathcal{P}_2$ on $\Spec(A_{\inf})\setminus\{s,\infty\}$.
Now set 
$$
\mathcal{P}:={\mathcal{P}_2}_{|\Spec(W(k)}
$$ 
as the restriction of $\mathcal{P}_2$ along the canonical morphism $A_{\inf}\to W(k)$ and $\alpha$ as the canonical trivialization
$$
\mathcal{P}_{|\Spec(W(k)[1/p]}\cong {\mathcal{P}_2}_{|\Spec(W(k)[1/p]}\cong {\mathcal{P}_0}_{|\Spec(W(k)[1/p]}.
$$
The data $(\mathcal{P},\alpha)$ defines a $k$-valued point in the Witt vector affine Grassmannian and we set
$$
\mathrm{sp}(\mathcal{P}^\prime,\alpha^\prime):=(\mathcal{P},\alpha). 
$$
This finishes the construction of $\mathrm{sp}$.
In a more compact form, the specialization map is given as the chain of equivalences and maps
$$
\begin{matrix}
G(B_{\dR}(C))/G(B_{\dR}^+(C))\\
\cong \{(\mathcal{G}-\textrm{torsor } \mathcal{P} \textrm{ on } \Spec(A_\inf)\setminus\{s\},\ \alpha \textrm{ a trivialization of } \mathcal{P}_{|\Spec(A_{\inf}[1/\xi])})\}\\
\cong \{(\mathcal{G}-\textrm{torsor } \mathcal{P} \textrm{ on } \Spec(A_\inf),\ \alpha \textrm{ a trivialization of }\mathcal{P}_{|\Spec(A_{\inf}[1/\xi])})\}\\
\to \{(\mathcal{G}-\textrm{torsor } \mathcal{P} \textrm{ on } \Spec(W(k)),\ \alpha^\prime \textrm{ a trivialization of }\mathcal{P}_{|\Spec(W(k)[1/p])})\}\\
\cong G(W(k)[1/p])/\mathcal{G}(W(k)).  
\end{matrix}
$$
Here the first and last $\cong$'s are the description of the affine Grassmannian via torsors (using Beauville-Laszlo, \Cref{lemma_beauville_laszlo_for_torsors}, for the first), the second equivalence is deduced from \Cref{sec:extend-tors-mathrmsp-1-main-theorem-a-e} (and \Cref{proposition:restricting-flat-schemes-is-fully-faithful}) and the arrow $\to$ is simply base change along $A_\inf\to W(k)$ (which maps the ideal $(\xi)$ to the ideal $(p)$).

Using this description it follows that the specialization map
$$
\mathrm{sp}\colon G(B_{\dR}(C))/G(B_{\dR}^+(C))\to G(W(k)[1/p])/\mathcal{G}(W(k))
$$
is equivariant for the action of the subgroup $G(A_{\inf}[1/\xi])\subseteq G(B_{\dR}(C))$ on the trivialization $\alpha$ of $\mathcal{P}_{|\Spec(A_{\inf}[1/\xi])}$.
For tori we can provide a different description of $\mathrm{sp}$.
Let $\mathcal{T}$ be a parahoric group scheme over $W(k)$ such that 
$$
T:=\mathcal{T}_{\Spec(W(k)[1/p]}
$$ 
is a torus. Then there are canonical bijections
$$
\mathrm{Gr}^{B^+_\dR}_{\mathcal{T}}(C)\cong X_\ast(T)
$$
(by observing that $B^+_\dR(C)$ is abstractly isomorphic to $C[[\xi]]$)
and
$$
\mathrm{Gr}^{W}_{\mathcal{T}}(k)\cong X_\ast(T)_\Gamma
$$
where $\Gamma$ is the absolute Galois group of $W(k)[1/p]$ (cf.\ \cite[Proposition 1.21.]{zhu_affine_grassmannians_and_the_geometric_satake_in_mixed_characteristic}).

\begin{lemma}
\label{lemma_specialization_for_tori}
For $\mathcal{T}$ as above the diagram
$$
\xymatrix{
\mathrm{Gr}^{B^+_\dR}_{\mathcal{T}}(C)\ar[d]^\cong\ar[r]^{\mathrm{sp}} & \mathrm{Gr}^{W}_{\mathcal{T}}(k)\ar[d]^\cong \\
X_\ast(T)\ar[r]^{\mathrm{can}}& X_\ast(T)_\Gamma
}
$$
with $\mathrm{can}\colon X_\ast(T)\to X_\ast(T)_\Gamma$ the canonical projection commutes.  
\end{lemma}
\begin{proof}
We first handle the case $T=\Gm$ (which implies $\mathcal{T}=\Gm$). Then $X_\ast(T)\cong \Z$ and $j\in \Z$ is mapped to the class of $\xi^j\in \mathrm{Gr}^{B^+_\dR}_{\mathcal{T}}(C)$. This class corresponds to the trivial line bundle $\mathcal{L}$ on $\Spec(A_\inf)\setminus\{s\}$ with trivialization $\xi^j$ on $\Spec(A_\inf[1/\xi])$. The line bundle $\mathcal{L}$ extends canonically to the trivial line bundle, again denoted $\mathcal{L}$, on $\Spec(A_\inf)$. Hence, the specialization map sends $\xi^j$ to the class in $\mathrm{Gr}^{W}_{\mathcal{T}}(k)$ corresponding to the pair 
$$
(W(k)=\mathcal{L}\otimes_{A_\inf}W(k),\bar{\xi}^j\colon W(k)[1/p]\cong W(k)[1/p]).
$$ 
But $\bar{\xi}^j=p^j$, which shows the claim for $T=\Gm$. As in \cite[Lemma 1.21]{richarz_affine_grassmannians_and_geometric_satake_equivalences} we can use this to deal with the case that $T$ is induced by using that $X_\ast(T)_\Gamma$ is torsionfree in such cases.
In the general case, choose a surjection
$$
T^\prime\to T
$$    
with $T^\prime$ induced and connected kernel $T^{\prime\prime}$.
As 
$$
\mathrm{Gr}_{T^\prime}^{B^+_\dR}(C)\twoheadrightarrow \mathrm{Gr}_{T}^{B^+_{\dR}}(C)
$$ 
is surjective (by Steinberg's theorem as $T^{\prime\prime}$ is connected, cf.\ \Cref{theorem_steinbergs_theorem}) the general case follows then from naturality of the specialization map. 
\end{proof}

We want to mention that Gleason studied specialization maps for $v$-sheaves much more thoroughly, cf.\ \cite{gleason2020specialization}.

\bibliography{biblio.bib}
\bibliographystyle{plain}

\end{document}